\documentclass[a4paper, 12pt]{article}

\usepackage{subfiles}
\usepackage[toc,page]{appendix}
\usepackage[bottom]{footmisc}
\usepackage{microtype}
\usepackage{tocbibind}
\usepackage{stmaryrd}   
\usepackage{amsmath}
\usepackage{verbatim}
\usepackage{mathrsfs}
\usepackage{mathtools}
\usepackage{amsthm} 
\usepackage{amssymb} 
\usepackage{faktor}
\usepackage{array}

\usepackage[colorlinks={true},linkcolor={blue},citecolor={blue}, urlcolor={blue} ]{hyperref}
\usepackage{braket}
\usepackage{transparent}
\usepackage{graphicx}
\usepackage{color}
\usepackage{rotating}
\usepackage{tikz}
\usepackage{tikz-cd}
\usepackage{authblk}
\usepackage[margin=2.5 cm]{geometry}
\usepackage{bm}

\newcommand{\uv}{\underline{v}}
\DeclareMathOperator{\uV}{\underline{\V}}

\DeclareMathOperator{\defeq}{\overset{\text{\tiny{def}}}{=}}

\DeclareMathOperator{\Z}{\mathbb{Z}}
\DeclareMathOperator{\Q}{\mathbb{Q}}
\DeclareMathOperator{\R}{\mathbb{R}}

\DeclareMathOperator{\V}{\mathbb{V}}

\DeclareMathOperator{\rad}{\text{rad}}

\DeclareMathOperator{\good}{\text{good}}
\DeclareMathOperator{\bad}{\text{bad}}
\DeclareMathOperator{\multi}{\text{multi}}
\DeclareMathOperator{\gen}{\text{gen}}
\DeclareMathOperator{\non}{\text{non}}

\DeclareMathOperator{\Vol}{\text{Vol}}
\DeclareMathOperator{\ur}{\text{ur}}

\DeclareMathOperator{\tmod}{\text{mod}}

\DeclareMathOperator{\Gal}{\text{Gal}}

\DeclareMathOperator{\GL}{\text{GL}}

\newcommand{\lagk}[1]{\langle #1 \rangle}
\newcommand{\lagl}{\langle l \rangle}

\numberwithin{equation}{section}

\theoremstyle{definition}

\newtheorem{thm}{Theorem}[section]
\newtheorem{prop}[thm]{Proposition}
\newtheorem{lem}[thm]{Lemma}
\newtheorem{cor}[thm]{Corollary}
\newtheorem{definition}[thm]{Definition}

\newtheorem{remark}[thm]{Remark}
\newtheorem{tiny-remark}{Remark}[thm]

\newtheorem{open-problem}[thm]{Open Problem}

\newtheorem*{thm-A}{Theorem A}
\newtheorem*{cor-B}{Corollary B}

\makeatletter

\title{The inter-universal Teichm\"uller theory and new Diophantine results over rational numbers. II}

\author{Zhong-Peng Zhou}
\date{}

\newcommand{\Addresses}{{
\bigskip

(Zhong-Peng Zhou) \textsc{Institute for Theoretical Sciences, Westlake University, 
No. 600 Dunyu Road,
Xihu District, Hangzhou,
Zhejiang, 310030, P.R. China}
\par\nopagebreak
\textit{E-mail address}: 
\texttt{zhouzhongpeng@westlake.edu.cn}
}}

\begin{document}

\maketitle

\begin{abstract}
[This is an older version of the paper, which will be updated soon.] 

In the present paper, we continue our research on the generalized Fermat equation $x^r + y^s = z^t$ with signature $(r, s, t)$, where $r, s, t \ge 2$ are positive integers such that $\frac{1}{r} + \frac{1}{s} + \frac{1}{t} < 1$.  
All known positive primitive solutions for the generalized Fermat equation when $\frac{1}{r} + \frac{1}{s} + \frac{1}{t} < 1$ are related to the Catalan solutions $1^n + 2^3 = 3^2$ and nine non-Catalan solutions.

By applying inter-universal Teichmüller theory and its slight modification in the case of elliptic curves over rational numbers, we deduce that the generalized Fermat equation $x^r + y^s = z^t$ has no non-trivial primitive solution except for those related to the Catalan solutions and nine non-Catalan solutions mentioned above, when $(r, s, t)$ is not a permutation of the following signatures:
\begin{itemize}
\item $(4,5,n)$, $(4,7,n)$, $(5,6,n)$, with $7 \le n \le 303$.
\item $(2,3,n)$, $(3,4,n)$, $(3,8,n)$, $(3,10,n)$, with $11\le n \le 109$ or $n\in \{113, 121\}$.
\item $(3,5,n)$, with $7\le n \le 3677$; $(3,7,n)$, $(3,11,n)$, with $11 \le n \le 667$.
\item $(3,m,n)$, with $13 \le m \le 17$, $m < n \le 29$; $(2,m,n)$, with $m \ge 5$, $n\ge 7$.
\end{itemize}

As a corollary, to solve the generalized Fermat equation $x^r + y^s = z^t$ with exponents $r,s,t \ge 4$, we are left with $244$ signatures $(r,s,t)$ up to permutation;
to solve the Beal conjecture, we are left with $2446$ signatures $(r,s,t)$ up to permutation.

\end{abstract}

\makeatletter
\@starttoc{toc}
\makeatother

\section*{Introduction}
\addcontentsline{toc}{section}{Introduction}
Let $ r, s, t \ge 2 $ be positive integers. The equation
\begin{equation} \label{GenFE} \tag{$\star$}
x^r + y^s = z^t, \;\text{with}\; x, y, z \in \mathbb{Z}
\end{equation}
is known as the generalized Fermat equation with signature $(r, s, t)$.
A solution $(x, y, z)$ to (\ref{GenFE}) is called non-trivial if $xyz \neq 0$, positive if $x, y, z \in \mathbb{Z}_{\ge 1}$, and primitive if $\gcd(x, y, z) = 1$.

Finding all the positive primitive solutions to (\ref{GenFE}) is a long-standing problem in number theory. The Pythagorean triples, which are the positive solutions to (\ref{GenFE}) with signature $(2, 2, 2)$, have been known since ancient times. Fermat's Last Theorem, proven by Wiles in 1994 \cite{Wiles}, states that (\ref{GenFE}) has no positive primitive solution with signature $(n, n, n)$ for $n \ge 3$. Catalan's conjecture, proven by Mih\u{a}ilescu in 2002 \cite{Mihilescu2004PrimaryCU}, states that if $x = 1$ in (\ref{GenFE}), then the only positive integer solution to $1 + y^s = z^t$ (with $s, t \ge 2$) is $(y, z, s, t) = (2, 3, 3, 2)$.

The behavior of positive primitive solutions to (\ref{GenFE}) is fundamentally determined by the size of the quantity
\[
\chi(r, s, t) = \frac{1}{r} + \frac{1}{s} + \frac{1}{t} - 1.
\]
Here, $\chi$ is the Euler characteristic of a certain stack associated with (\ref{GenFE}).

For the spherical case where $\chi > 0$, $(r, s, t)$ is a permutation of $(2, 2, t)$, $t \ge 2$, or $(2, 3, t)$, $3 \le t \le 5$. In this case, parametrizations for positive primitive solutions for these signatures can be found in Beukers \cite{spherical_case1}, Edwards \cite{spherical_case2}, and Cohen \cite{Cohen-GTM240}, Section 14.

For the Euclidean case where $\chi = 0$, $(r, s, t)$ is a permutation of $(2, 3, 6)$, $(2, 4, 4)$, or $(3, 3, 3)$. In this case, non-trivial primitive solutions for these signatures come from the Catalan solution $1^6 + 2^3 = 3^2$, cf. \cite{Bennett-Chen-Dahmen-Yazdani}, Proposition 6 for a proof.

For the hyperbolic case where $\chi < 0$, the number of primitive solutions for each fixed signature $(r, s, t)$ is finite, cf. Darmon-Granville \cite{Darmon-Granville}. 
For this case, only the Catalan solution $1^n + 2^3 = 3^2$ and the following nine non-Catalan solutions are currently known:
\begin{gather*}
    2^5 + 7^2 = 3^4, \quad 7^3 + 13^2 = 2^9, \quad
    2^7 + 17^3 = 71^2, \quad 3^5 + 11^4 = 122^2, \\  
    17^7 + 76271^3 = 21063928^2, \quad	
    1414^3 + 2213459^2 = 65^7, \quad
    9262^3 + 15312283^2 = 113^7, \\ 
    43^8 + 96222^3 = 30042907^2, \quad 33^8 + 1549034^2 = 15613^3.
\end{gather*}
Since all known solutions have $\min\{r, s, t\} \le 2$, Beal conjectured in 1993 \cite{Beal_Conj} that the generalized Fermat equation (\ref{GenFE}) has no positive primitive solution when $r, s, t \ge 3$.
This conjecture is known as the Beal conjecture, also known as the Mauldin conjecture and the Tijdeman-Zagier conjecture.

For many signatures $(r,s,t)$ with $\chi \le 0$, the generalized Fermat equation $x^r + y^s = z^t$ has been proved to have no other non-trivial primitive solution, except for the possible solutions related to the Catalan solution and the nine non-Catalan solutions listed above.
Some of these signatures needed in this paper are listed below, which is used to avoid repeated computation for some proven cases during the search for possible solutions:

\begin{itemize}
\setlength{\itemsep}{0pt}
\item $(n, n, n), n \ge 3$, cf. Wiles \cite{Wiles}, Taylor-Wiles \cite{Taylor-Wiles}, Mochizuki-Fesenko-Hoshi-Minamide-Porowski \cite{ExpEst}.

\item $(n, n, 2), n \ge 4$; $(n, n, 3), n \ge 3$, cf. Darmon-Merel \cite{Darmon-Merel}, Poonen \cite{Poonen}.

\item $(2, 3, n), n \in \{7, 8, 9, 10, 15\}$, cf. Poonen-Schaefer-Stoll \cite{Poonen-Schaefer-Stoll}, Bruin \cite{Bruin-16, Bruin-19}, Brown \cite{Brown_case_2_3_10}, Siksek-Stoll \cite{Siksek-Stoll_case_2_3_15}.

\item $(2, 4, n), n \ge 4$, cf. Ellenberg \cite{Ellenberg}, Bennett-Ellenberg-Ng \cite{Bennett-Ellenberg-Ng};
$(2, n, 4), n \ge 4$, Corollary of Bennett-Skinner \cite{Bennett-Skinner-8}, Bruin \cite{Bruin-18}.

\item $(2, 6, n), n \ge 3$, cf. Bennett-Chen \cite{Bennett-Chen};
$(2, n, 6), n \ge 3$, cf. Bennett-Chen-Dahmen-Yazdani \cite{Bennett-Chen-Dahmen-Yazdani}.

\item $(3, 3, n), n\ge 3$, cf. Chen-Siksek \cite{Chen-Siksek}, Zhou \cite{Diophantine_after_IUT_I}.

\item $(3, 4, 5)$, $(5, 5, 7)$, $(5, 7, 7)$, cf. Siksek-Stoll \cite{Siksek-Stoll}, Dahmen-Siksek \cite{Dahmen-Siksek}.

\item $(5, 5, q)$, primes $q \ge 11 > 3 \sqrt{5 \log_2(5)}$, cf. Bartolomé-Mih\u{a}ilescu \cite{semi_local_approximation_fermat_catalan_popular}.
\end{itemize}

In 2022, Mochizuki, Fesenko, Hoshi, Minamide, and Porowski \cite{ExpEst} established a $\mu_6$-version of Mochizuki's inter-universal Teichm\"uller theory (IUT theory), cf. \cite{IUTchI, IUTchII, IUTchIII, IUTchIV}, which managed to ``treat'' bad places that divide the prime $2$. They then verified the numerically effective versions of the Vojta, ABC, and Szpiro Conjectures over the rational numbers and imaginary quadratic number fields. As a result, they verified that Fermat’s Last Theorem (FLT) holds for prime exponents $> 1.615 \cdot 10^{14}$ — which is sufficient, combined with the results of Vandiver, Coppersmith, and Mih\u{a}ilescu-Rassias, to yield an unconditional new alternative proof of Fermat’s Last Theorem.

For the first time proving effective abc inequalities and containing a new proof of FLT, the paper \cite{ExpEst} sheds light on solving generalized Fermat equations. In the previous paper \cite{Diophantine_after_IUT_I}, by making slight modifications to the estimation in \cite{IUTchIV} and \cite{ExpEst}, we reduced the coefficient of the effective abc inequality presented in \cite{ExpEst}, and performed a preliminary study on the generalized Fermat equation $x^r + y^s = z^t$ with exponents $r, s, t \ge 3$.

In the present paper, in order to further study more situations, the signatures with $\chi < 0$ are divided into four cases. For each of these cases, we can use different classes of [modified] Frey-Hellegouarch curves in IUT theory to obtain certain upper bounds related to the non-trivial primitive solutions to (\ref{GenFE}).

Several methods are used in \cite{Diophantine_after_IUT_I} and then are generalized in the present paper, including:
\begin{itemize}
\item \textbf{Theoretical mathematical methods}:
(i) Primarily using slightly changed IUT theory, such as $2$-torsion initial $\Theta$-data, sharper effective abc inequalities over rational numbers and their applications to the generalised Fermat equations.
(ii) A little from analytic number theory related to prime number theorem, for upper bounds related to solutions, cf. the estimation in \cite{Diophantine_after_IUT_I}, Section 2.3 for an example.
\item \textbf{Computations using computers}: (i) Reducing upper bounds in many specific cases. 
For example, for $r,s,t \ge 8$ and positive primitive solutions $(x,y,z)$ to $x^r + y^s = z^t$, let $h \defeq \log(x^r y^s z^t)$. 
By using analytic number theory and effective abc inequalities, it is shown in \cite{Diophantine_after_IUT_I}, Theorem 4.4 that  $h < C_0 \defeq 10^6$. Then we can choose finitely concrete examples of $S$ and $k$ in \cite{Diophantine_after_IUT_I}, Lemma 4.3, to show that $h$ does not belong to the interval $(600, C_0)$. As a result, we can deduce that $h \le 600$. (ii) Searching for possible solutions via computer for large exponents, cf. \cite{Diophantine_after_IUT_I}, Corollary 4.6 for an example. In this paper, new algorithms based on thestructure of solutions are used to accelerating the speed of search.
\end{itemize}

It is worth noting that by choosing more concrete examples of $S$ and $k$ in \cite{Diophantine_after_IUT_I}, Lemma 4.3, we can show that $h$ does not belong to the interval $(600, C_0)$ for arbitrary large $C_0 > 0$. Since we can always obtain a [possibly large] upper bound $C_0$ for $h$ using analytic number theory, the value of $C_0$ is not the core of the proof. Therefore, we shall omit the process of proving ``$h < C_0$'' in the several cases in this paper, and claim that we can always do the similar things presented above to provide ``a valid and rigorous full proof”, cf. Remark \ref{1-3-rmk:  local-global inequalities} for a discussion about this.

The main results of the present paper are as following:
\begin{thm-A}
The generalized Fermat equation $x^r + y^s = z^t$ has no non-trivial primitive solution except for the Catalan solutions and the nine non-Catalan solutions enumerated earlier, when $(r, s, t)$ does not correspond to any permutation of the following cases:
\begin{itemize}
\item $(4,5,n)$, $(4,7,n)$, $(5,6,n)$, with $7 \le n \le 303$.
\item $(2,3,n)$, $(3,4,n)$, $(3,8,n)$, $(3,10,n)$, with $11\le n \le 109$ or $n\in \{113, 121\}$.
\item $(3,5,n)$, with $7\le n \le 3677$; $(3,7,n)$, $(3,11,n)$, with $11 \le n \le 667$.
\item $(3,m,n)$, with $13 \le m \le 17$, $m < n \le 29$; $(2,m,n)$, with $m \ge 5$, $n\ge 7$.
\end{itemize}
\end{thm-A}

\begin{cor-B}
To solve the generalized Fermat equation $x^r + y^s = z^t$ with exponents $r,s,t \ge 4$, we are left with $244$ signatures $(r,s,t)$ up to permutation;
to solve the Beal conjecture, we are left with $2446$ signatures $(r,s,t)$ up to permutation.
\end{cor-B}

It is worth noting that Catalan's conjecture is used in this paper to get upper bounds for the exponents; with further analysis of upper bounds and more computations for searching solutions, an expanded set of signatures including the permutations of $(4, 7, n)$, $(5, 6, n)$ are likely to be solved.

The present paper is organized as follows. In Section 1, we make preparations for the application of IUT Theory in three aspects: the ``$2$-torsion version'' of IUT theory, the arithmetic of elliptic curves for the construction of initial $\Theta$-data, and the procedure to deduce ``upper bounds'' from ``partial inequalities''. 
Such procedure is a generalization of the methods in \cite{Diophantine_after_IUT_I}.

 In Sections 2-4, we research on the generalized Fermat equation (\ref{6-GenFE}) with each of the above four cases of signatures, using some classes of [modified] Frey-Hellegouarch curves:
\begin{itemize}
\item For general signatures $(r, s, t)$ with $r, s, t \ge 4$, the curve $Y^2 + XY = X^3 + \frac{b-a-1}{4}\cdot X^2 - \frac{ab}{16}\cdot X$, 
where $(a,b,c)$ is a triple of non-zero coprime integers such that $a+b = c$, $4\mid (a+1)$ and $16\mid b$, cf. Section 2 for details.
\item For permutations of $(2, 3, t)$ with $t \ge 7$, the curve $Y^2 = X^3 + 3bX + 2a$, 
where $(a,b,c)$ is a triple of non-zero coprime integers such that $a^2 + b^3 = c$, cf. Section 3 for details.
\item For permutations of $(3, r, s)$ with $r \ge 3$, $s \ge 4$, the curve $Y^2 + 3cXY + aY= X^3$,
where $(a,b,c)$ is a triple of non-zero coprime integers such that $a + b = c^3$, cf. Section 4 for details.
\item For permutations of $(2, r, s)$ with $r \ge 4$, $s \ge 5$, the curve $Y^2 = X^3 + 2cX^2 + aX$, 
where $(a,b,c)$ is a triple of non-zero coprime integers such that $a + b = c^2$, the related work will be undertaken in future studies.
\end{itemize}

Section 5 is a summary of the conclusions.

\section*{Acknowledgements}
The main results of the present paper were obtained during the author's stay at Westlake University from September 2024 to February 2025.
In particular, the discussions with Ivan Fesenko and later with Shinichi Mochizuki about the reason of adding 3-torsion points in the definition of initial $\Theta$-data were crucial to the completion of this paper. 
The author also benefited from the talks given by Arata Minamide, Emmanuel Lepage, and Preda Mihăilescu at Westlake University, which provided some inspiration for this work. 
The author is especially grateful to Ivan Fesenko for his encouragements, discussions, and support during his stay, which greatly surpassed the author's expectations regarding the support received.

 \setcounter{section}{-1}
 \section{Notations and Conventions}

\hfill \break
\noindent\textbf{Numbers:}

Let $N$ be a non-zero integer,  $k$ be a positive integer. For each prime number $p$, write $v_p(N)$ for the $p$-adic valuation of $N$.
The radical of $N$, denoted by $\rad(N)$, is the product of distinct prime divisors of $N$;
we shall write $N_{(k)}$ for the ``coprime to $k$ part'' of $N$, 
and write $N^{\lagk{k}}$ for the product of $p^{v_p(N)}$ for all possible prime number $p$ such that $v_p(N)$ is divided by $k$. 
Hence we have
\begin{align*}
\rad(N) \defeq \prod_{p:\, p\mid N} p, \quad 
N_{(k)} \defeq \prod_{p:\, p\nmid k} p^{v_p(N)}, \quad
N^{\lagk{k}} \defeq \prod_{p:\, k\mid v_p(N)} p^{v_p(N)} . 
\end{align*}

Let $p$ be a prime number.
We shall write $\Q$ (respectively, $\R$; $\Q_p$) for the field of rational numbers (respectively, real numbers; $p$-adic rational numbers), 
and write $\Z$ for the ring of rational integers.
We shall call finite extension of $\Q$ number field, call finite extension of $\Q_p$ $p$-adic local field.
We shall write $\Z_{\ge 1}$ for the set of positive integers.

\hfill \break
\noindent\textbf{Curves:}

Let $K$ be a field of characteristic zero with an algebraic closure $\overline{K}$,
$E$ be an elliptic curve defined over $K$, $n\ge 1$ be an positive integer. 
Then we shall write $E[n]$ for the group of $n$-torsion points of $E_{\overline{K}}\defeq E\times_K \overline{K}$, and write $K(E[n])\subseteq \overline{K}$ for the ``$n$-torsion point field" of $E$, 
i.e. the field generated by $K$ and the coordinates of the points in $E[n]$.        

\hfill \break
\noindent\textbf{Ramification Datasets:}

Ramification dataset is a concept introduced in \cite{Diophantine_after_IUT_I}, Section 1.3, which encodes the ramification information of initial $\Theta$-data.
This concept will be used frequently in the present paper.

\section{Preparations for IUT Theory Applications}

\subsection{$2$-torsion version of IUT theory}

We shall introduce the concept of $2$-torsion initial $\Theta$-data.

\begin{definition} \label{1-def: mu_2-initial Theta-Data}
We shall refer to $2$-torsion initial $\Theta$-data as any collection of data
\begin{align*}
(\overline{F}/F,X_F,l,\underline{C}_K, \uV,\V_{\tmod}^{\bad}, \underline{\epsilon})
\end{align*}
satisfying the following conditions:
\begin{itemize}
\item \cite{IUTchI}, Definition 3.1, (a), (c), (d), (e), (f).
\item The ``$2$-torsion version'' of \cite{IUTchI}, Definition 3.1, (b), i.e., the condition obtained by replacing, in \cite{IUTchI}, Definition 3.1, (b), ``$2 \cdot 3$-torsion points of $E_F$ are rational over $F$'', by
``2-torsion points of $E_F$ are rational over $F$, and $E_F$ has a model over $F_{\tmod}$''.
\end{itemize}
\end{definition}

\begin{tiny-remark} \label{1-1-rmk1}
It is worth noting that the results of \cite{IUTchI,IUTchII,IUTchIII,IUTchIV} still hold for its $2$-torsion version, i.e. by replacing initial $\Theta$-data with $2$-torsion initial $\Theta$-data. 
In these papers, the condition ``$3$-torsion points of $E_F$ are rational over $F$'', i.e. ``$F(E[3])=F$'' is only used in \cite{IUTchI}, Remark 3.1.5, \cite{IUTchIV}, Theorem 1.10 and \cite{IUTchIV}, Corollary 2.2 via \cite{IUTchIV}, Proposition 1.8, (iv), (v).

As a consequence of ``$F(E[3])=F$'', it is stated that $K$ is Galois over $F_{\tmod}$ at the beginning of \cite{IUTchI}, Remark 3.1.5.
This still holds for its $2$-torsion version. 
Since $E_F$ has a model $E$ defined over $F_{\tmod}$, we can put $L = F_{\tmod}(E[l])$, then $L/F_{\tmod}$ is Galois. Since $F/F_{\tmod}$ is Galois by the definition of [$2$-torsion ]initial $\Theta$-data, we can see that $K = F(E_F[l]) = F\cdot L$ is Galois over $F_{\tmod}$.

The condition ``$F(E[3])=F$'' is also used to show that $E_F$ has a model over $F_{\tmod}$ in \cite{IUTchIV}, Theorem 1.10 and \cite{IUTchIV}, Corollary 2.2, after an initial $\Theta$-data is constructed. 
This is one of the conditions in the definition of $2$-torsion initial $\Theta$-data.

Hence the results of \cite{IUTchI,IUTchII,IUTchIII,IUTchIV} still hold for its $2$-torsion version.
\end{tiny-remark}

\begin{tiny-remark} \label{1-1-rmk2}
By applying Definition \ref{1-def: mu_2-initial Theta-Data},
the $2$-torsion versions of Section 1.2 and Section 1.3 of \cite{Diophantine_after_IUT_I} remains valid without essential changes, with the following exceptions:

(i) All the ``$\mu_6$-initial $\Theta$-data'' should be replaced by ``$2$-torsion initial $\Theta$-data''.

(ii) In the $2$-torsion version of \cite{Diophantine_after_IUT_I}, Lemma 1.6, the condition ``$F=F(E_F[6])$'' should be replaced by ``$F=F(E_F[2])$'', 
and the condition ``$\V_{\tmod}^{\bad} \subseteq \V_{\tmod}^{\non}$ is a nonempty set of nonarchimedean valuations of $F_{\tmod}$'' should be replaced by ``$\V_{\tmod}^{\bad} \subseteq \V_{\tmod}^{\non}$ is a nonempty set of nonarchimedean valuations of $F_{\tmod}$ of odd residue
characteristic''.

(iii) The proof of the $2$-torsion version of \cite{Diophantine_after_IUT_I}, Proposition 1.9 follows from the discussion in Remark \ref{1-1-rmk1}.
Note that the $2$-torsion version of \cite{Diophantine_after_IUT_I}, Corollary 1.13 also holds, since we only use the condition ``$F=F(E_F[2])$'' rather than the condition ``$F=F(E_F[6])$'' in the proof of \cite{Diophantine_after_IUT_I}, Corollary 1.13. 

(iv) For a $\mu_6$- or $2$-torsion initial $\Theta$-data $\mathfrak D  =(\overline{F}/F,X_F,l,\underline{C}_K, \uV,\V_{\tmod}^{\bad}, \underline{\epsilon})$ with $F_{\tmod} = \Q$, 
we shall say $\mathfrak D$ is \textbf{of type} $\bm{(l, N, N')}$ if $N$ is the denominator of the $j$-invarient of $j(E_F) \in F_{\tmod} = \Q$, and $N' = \prod_{p:\, v_p \in \V_{\tmod}^{\bad}} p^{\max\{0, v_p(N) \} }$.

\begin{cor} \label{1-cor: The log-volume of ramification dataset}
Let $\mathfrak{R}$ be a ramification dataset with base prime $l\ge 5$ and base index $e_0$. Then there exists an algorithm to compute a real number $\Vol(\mathfrak{R}) \ge 0$ that depends only on $\mathfrak{R}$, 
such that for any $\mu_6$-initial $\Theta$-data or $2$-torsion initial $\Theta$-data $\mathfrak{D}$ of type $(l,N,N')$ admitting $\mathfrak{R}$, we have
\begin{equation*}\small
\frac{1}{6}\log(N') \le  \frac{l^2+5l}{l^2+l-12}\cdot \big( (1-\frac{1}{e_0  l})\cdot \log\rad(N) - \frac{1}{e_0}(1- \frac{1}{l}) \log\rad(N^{\lagl})  \big) + \Vol(\mathfrak{R}),
\end{equation*}

\end{cor}
\begin{proof}
The corollary follows from \cite{Diophantine_after_IUT_I}, Corollary 1.13 and Remark \ref{1-1-rmk2}, (iii).
\end{proof}

We shall also provide a $2$-torsion version of \cite{Diophantine_after_IUT_I}, Proposition 2.5 as following.
\end{tiny-remark}

\begin{prop} \label{1-prop: construction of mu_2 initial Theta-data}
Let $E$ be an elliptic curve defined over $\Q$ with $j$-invariant $j(E) \in \Q$;
$N$ be the denominator of $j(E)$;
$F$ be a number field which is Galois over $\Q$;
$l\ge 5$ be a prime number such that $l\nmid [F:\Q]$;
$X_F$ be the punctured elliptic curve associated to $E_F \defeq E \times_{\Q} F$.
Suppose that:
\begin{itemize}
\item[(a)] $\sqrt{-1}\in F$, $F(E[2]) = F$, $E_F$ is semi-stable, and $F \subseteq \Q(E[n])$ for some positive integer $l\nmid n$. 
\item[(b)] $j(E)\notin \{0,2^6\cdot 3^3,2^2\cdot 73^3\cdot 3^{-4},2^{14}\cdot 31^3\cdot 5^{-3}\}$.
\item[(c)] We have $l\ge 23$ and $l\neq 37,43,67,163$; or $E$ is semi-stable and $l\ge 11$; or $N$ is not a power of $2$, $l\ge 11$ and $l\neq 13$.
\item[(d)] We have $N'_{2l} \neq 1$, where $N'_{2l} \defeq N_{(2l)} / N_{(2l)}^{\lagl} = \prod_{p:\, p\nmid 2l,\, l\nmid v_p(N)} p^{v_p(N)}$.
\end{itemize} 
Then there exists a $2$-torsion initial $\Theta$-data $$\mathfrak D(E,F,l,2\text{-tor}) = (\bar{F}/F,X_F,l,\underline{C}_K, \uV,\V_{\tmod}^{\bad}, \underline{\epsilon}),$$ 
with $\log(\mathfrak q) = \log(N'_{2l})$ [cf. the notation of the $2$-torsion version of \cite{Diophantine_after_IUT_I}, Definition 1.7, (ii)].
\end{prop}
\begin{proof}
By making use of the $2$-torsion version of \cite{Diophantine_after_IUT_I}, Lemma 1.6, the proof of Proposition \ref{1-prop: construction of mu_2 initial Theta-data} is similar to the proof of \cite{Diophantine_after_IUT_I}, Proposition 2.5, except that we shall define $\V_{\tmod}^{\bad} \defeq \{v_p \in \V_{\Q}^{\non}: p\nmid 2l,\; l \nmid v_p(N) \}$ and replace ``$N'_l$'' by ``$N'_{2l}$''.
\end{proof}

\subsection{Arithmetic of elliptic curves}

\begin{lem} \label{1-2-lem: basic properties}
Let $p$ be a prime number, $E$ be an elliptic curve defined over $\Q_p$ with $j$-invariant $j(E) \in \Q_p$.
For each positive integer $n$, write $K_n \defeq \Q_p(E[n])$, and write $e_p(n)$ for the ramification index of $K_n$ over $\Q_p$.

(i) If $v_p(j(E)) \ge 0$, then $E$ has potentially good reduction;
if $v_p(j(E)) < 0$, then $E$ has potentially multiplicative reduction;
for each $n\ge 1$, we have $\mu_n \subseteq K_n$, and $K_n$ is Galois over $\Q_p$,
where $\mu_n$ represents the $n$-th roots of unity;
for coprime integers $m,n$, we have $e_p(m), e_p(n) \mid e_p(mn)$, and $e_p(mn) \mid  e_p(m) \cdot e_p(n)$.

(ii) We have $(p-1) \mid e_p(p)$ and $e_p(n) \mid \#\GL(2,\Z/n\Z)$ for $n\ge 1$.
If $E$ has potentially good reduction, then for $n\ge 1$ such that $p\nmid n$, we have $e_2(n) \mid 24$ for $p = 2$; $e_3(n) \mid 12$ for $p = 3$, and $e_p(n) \mid 6$ for $p\ge 5$.
If $E$ has potentially multiplicative reduction, then for $n\ge 1$ such that $p\nmid n$,  we have $e_p(n) \mid 2n$, $e_p(pn) \mid 2(p-1)n$ for $p\ge 2$, and $e_2(4n) \mid 8n$.

(iii) If $n \ge 3$ and $p\nmid n$, then $E_{K_n} \defeq E\times_{\Q_p} K_n$ has semi-stable reduction, i.e. has good or multiplicative reduction.

(iv) If $E$ has good reduction, then $e_p(p) \in \{p-1, p(p-1), p^2-1\}$, and for $p\nmid n$, we have $e_p(n) = 1$; 
if $E$ has multiplicative reduction, then $e_p(p) = (p-1)\cdot p / \gcd(p,v_p(j(E)) \in \{p-1, p(p-1)\}$, and for $p\nmid n$, we have $e_p(n) = n / \gcd(n,v_p(j(E)) \mid n$.
\end{lem}
\begin{proof}
 For finite extension of $p$-adic local fields $L/M$, we shall write $e(L/M)$ for the ramification index of $L$ over $M$.

For assertion (i), the results related to reduction type follows from the arithmetic of elliptic curves, cf. \cite{Silverman2009EllipticCurves}, Proposition 5.5 for a reference. 
Since $\gcd(m,n) = 1$, we have $K_{mn} = \Q_p(E[mn]) = \Q_p(E[m], E[n]) = K_m\cdot K_n$, then the results related to ramification indices follows from algebraic number theory.

For assertion (ii), since $\mu_p \subseteq K_p$, we have $p-1 = e(\Q_p(\mu_p)/\Q_p) \mid e_p(p)$; since $\Gal(K_n/\Q_p)$ is isomorphic to a subgroup of $\GL(2,\Z/n\Z)$, we have $e_p(p) \mid [K_n/\Q_p] \mid \#\GL(2,\Z/n\Z)$. 
Now let $n$ be a positive integer such that $p\nmid n$, and suppose that $E$ has potentially good reduction.
Then by \cite{Serre1971PropritsGD}, 5.6, there exists a finite extension $L$ of the maximal unramified field extension $\Q_p^{\ur}$ of $\Q_p$, such that $E$ has good reduction over $L$. Moreover, we have $[L:\Q_p^{\ur}] \mid 24$ if $p=2$, $[L:\Q_p^{\ur}] \mid 12$ if $p=3$, and $[L:\Q_p^{\ur}] \mid 6$ if $p\ge 5$.
Then since $p\nmid n$, we have $e(L(E[n]) / L) = 1$. Then $e_p(n) = e(K_n/\Q_p) = e(\Q_p^{\ur}(E[n]) / \Q_p^{\ur} ) \mid e(L(E[n])/ \Q_p^{\ur}) = e(L/\Q_p^{\ur}) \mid [L:\Q_p^{\ur}]$, hence we have $e_p(n) \mid 24$ if $p=2$, etc. The case that $E$ has potentially multiplicative reduction can be proved similarly.

For assertion (iii), we only need to prove for the case when $n = p \ge 3$ is an odd prime number and for $n=4$. For the case $n=p$, cf. \cite{IUTchIV}, Proposition 1.8 $(v)$.
For the case $n=4$, as is explained by Chris Wuthrich, we can prove by contradiction. 
Suppose that there exist $E$ which has additive reduction over $K_n = \Q_p(E[4])$, let $\mathcal{E}$ be the N\'eron model of $E$ over $\mathcal O_{K_n}$. 
Then sicne $p\nmid n=4$, we have $p\ge 3$.
By the Kodaira classification, the group of components of the Néron model $\mathcal{E}/\mathcal{O}_{K_n}$ cannot contain $\mathbb{Z}/4\mathbb{Z}\times \mathbb{Z}/4\mathbb{Z}$. Therefore, we would have a point of order $2$ in the identity component $\mathcal{E}^0(\mathcal{O}_{K_n})$. The reduction is additive, so the group of non-singular points is a $p$-group, which means that our $2$-torsion  point must belong to the kernel of reduction, i.e.,  the formal group. However the torsion subgroup of a formal group over $\mathcal{O}_{K_n}$ is also a $p$-group --- a contradiction! This proves assertion (iii).

assertion (iv) is a reformulation of \cite{Diophantine_after_IUT_I}, Proposition 1.8, (iii) and (iv).
\end{proof}

For the convenience of description, we shall introduce some definitions.

\begin{definition}  \label{1-2-def}
Let $E$ be an elliptic curve defined over $\Q$, and let $F$ be a finite Galois field extension of $\Q$.
Let $l\ge 5$ be a prime number, such that $l\nmid [F:\Q]$.
For each prime number $p$, choose a place $\uv$ of $K\defeq F(E[l])$ dividing $p$, and write $e_p(E,F,l)$ for the ramification index of $K_{\uv}$ over $\Q_p$. 
\end{definition}
\begin{tiny-remark}
Since $\Q(E[l])$ is Galois over $\Q$ [cf. \cite{TorsionPointFields}, Proposition 5.2.1 and Proposition 5.2.2] and $F$ is Galois over $\Q$, we can see that $K$ is Galois over $\Q$. Hence $e_p(E,F,l)$ is independent to the choice of $\uv$, i.e.  $e_p(E,F,l)$ is well-defined.
\end{tiny-remark}

\begin{cor} \label{1-2-cor0}
Let $E$ be an elliptic curve defined over $\Q$, such that $E$ has semi-stable reduction at each $p\neq 2,3$.
Let $N$ be the denominator of the $j$-invariant of $E$.
Let $l\ge 5$ be a prime number, and we shall follow the notation of Definition \ref{1-2-def}.

Write $F\defeq \Q(E[12])$, $E_F \defeq E\times_{\Q} F$, and let $w$ be a place of $F$ with residue characteristic $p$. 
Then if $p \nmid N$, $E_F$ has good reduction at $w$; if $p \mid N$, $E_F$ has multiplicative reduction at $w$; we have $\sqrt{-1} \in F$, $2\mid e_2(E,F,l)$, $2\mid e_3(E,F,l)$ and $(l-1) \mid e_l(E,F,l)$.

In the case of $p \nmid N$, we have 
$e_2(E,F,l) \mid 2^8\cdot 3^2$ for $p=2$, $e_3(E,F,l) \mid 2^6\cdot 3^2$ for $p=3$, 
$e_l(E,F,l) \in \{l-1, l(l-1), l^2-1\}$ for $p=l$, and $e_p(E,F,l) = 1$ for $p\neq 2,3,l$.
 
In the case of $p \mid N$, we have
$e_2(E,F,l) \mid 24l$ for $p=2$, $e_3(E,F,l) \mid 48l$ for $p=3$, 
$e_l(E,F,l) \mid 12 l (l-1)$ for $p=l$, 
and $e_p(E,F,l) \mid 12l$ for $p\neq 2,3,l$.
\end{cor}
\begin{proof}
We shall follow the notation of Definition \ref{1-2-def} and prove via Lemma \ref{1-2-lem: basic properties}.
Since $12$-torsion points of $E_F$ are defined over $F$, we can see that $E_F$ has good [if $p\nmid N$] or multiplicative [if $p\mid N$] reduction at $w$, $\sqrt{-1}\in \mu_4 \subseteq \Q(E[4])$, and for $p\in\{2,3,l\}$, we have $(p-1) \mid e_p(E,F,l)$, cf.  Lemma \ref{1-2-lem: basic properties}.
Since $\Q(\sqrt{-1})$ only ramifies at $2$, by Lemma \ref{1-2-lem: basic properties}, we have the following proof:

In the case of $p\nmid N$, 
for $p=2$, we have $e_2(E,F,l) \mid e_2(4) \cdot e_2(3l)$, which divides $\#\GL(2,\Z/4\Z) \cdot 24 = 2^8 \cdot 3^2$;
for $p=3$, we have  $e_3(E,F,l) \mid e_3(3) \cdot e_3(4l)$, which divides $\#\GL(2,\Z/3\Z) \cdot 12 = 2^6 \cdot 3^2$;
for $p=l$, we have $e_l(E,F,l) = e_l(12l) = e_l(l) \in \{ l-1, l(l-1), l^2-1\}$;
and for $p\neq 2,3,l$, we have $e_p(E,F,l) = e_p(12l) = 1$.

In the case of $p\mid N$, 
for $p=2$, we have $e_2(E,F,l) = e_2(4\cdot 3l) \mid 8 \cdot 3l = 24l$;
for $p=3$, we have $e_3(E,F,l) = e_3(3\cdot 4l) \mid 2\cdot (3-1)\cdot 3 \cdot 4l = 48l$;
for $p=l$, since $E$ has multiplicative reduction at $l$, we have $e_l(E,F,l) = e_l(12l) \mid 12 l (l-1)$;
and for $p\neq 2,3,l$, since $E$ has multiplicative reduction at $p$, we have $e_p(E,F,l) = e_p(12l) \mid 12l$.
\end{proof}

\begin{cor} \label{1-2-cor1}
Let $(a,b,c)$ be a triple of non-zero coprime integers such that $a+b+c = 0$, $4\mid (a+1)$ and $16\mid b$.
Let $E$ be the elliptic curve defined over $\Q$ by the equation 
$$Y^2 + XY = X^3 + \frac{b-a-1}{4}\cdot X^2 - \frac{ab}{16}\cdot X .$$
Let $l\ge 5$ be a prime number, and we shall follow the notation of Definition \ref{1-2-def}.

(i) $E$ has semi-stable reduction [i.e. good or multiplicative reduction] at each $p$;
$\Q(E[2]) = \Q$; the $j$-invariant of $E$ is $j(E) = \frac{(a^2+ab+b^2)^3}{2^{-8}a^2b^2c^2}$,
 whose denominator $N$ equals $2^{-8}a^2b^2c^2$.

(ii) Write $F\defeq \Q(\sqrt{-1}, E[3])$, $E_F \defeq E\times_{\Q} F$, and let $w$ be a place of $F$ with residue characteristic $p$. 
Then if $p \nmid N$, $E_F$ has good reduction at $w$; if $p \mid N$, $E_F$ has multiplicative reduction at $w$; we have $2\mid e_2(E,F,l)$, $2\mid e_3(E,F,l)$ and $(l-1) \mid e_l(E,F,l)$.

In the case of $p \nmid N$, we have 
$e_2(E,F,l) = 2$ for $p=2$, $e_3(E,F,l) \in \{2,6,8\}$ for $p=3$, $e_l(E,F,l) \in \{l-1, l(l-1), l^2-1\}$ for $p=l$, and $e_p(E,F,l) = 1$ for $p\neq 2,3,l$.

In the case of $p \mid N$, we have
$e_2(E,F,l) \mid 6l$ for $p=2$, $e_3(E,F,l) \mid 6l$ for $p=3$, $e_l(E,F,l) \mid 3l(l-1)$ for $p=l$, and $e_p(E,F,l) \mid 3l$ for $p\neq 2,3,l$.

(iii) Write $F\defeq \Q(\sqrt{-1})$, $E_F \defeq E\times_{\Q} F$, and let $w$ be a place of $F$ with residue characteristic $p$. 
Then if $p \nmid N$, $E_F$ has good reduction at $w$; if $p \mid N$, $E_F$ has multiplicative reduction at $w$; we have $2\mid e_2(E,F,l)$ and $(l-1) \mid e_l(E,F,l)$.

In the case of $p \nmid N$,  we have 
$e_2(E,F,l) = 2$ for $p=2$, $e_l(E,F,l) \in \{l-1, l(l-1), l^2-1\}$ for $p=l$, and $e_p(E,F,l) = 1$ for $p\neq 2, l$.

In the case of $p \mid N$, we have
$e_2(E,F,l) \mid 2l$ for $p=2$, 
$e_l(E,F,l) \mid l (l-1)$ for $p=l$, 
and $e_p(E,F,l) \mid l$ for $p\neq 2, l$.
\end{cor}
\begin{proof}
For assertion (i), by using Tate's algorithm, we have 
\small\begin{align*}
c_4(E) = a^2 + ab + b^2, \quad
\Delta(E) = 2^{-8}a^2b^2c^2, \quad 
j(E) = \frac{(a^2+ab+b^2)^3}{2^{-8}a^2b^2c^2}.
\end{align*}\normalsize
Since $\gcd(a,b,c) = 1$ and $16\mid b$, we have $\gcd(c_4(E), \Delta) = 1$, hence $E$ has semi-stable reduction at each $p$, and $N=2^{-8}a^2b^2c^2$. 
Since $E$ can also be defined by the equation $Y^2 = X(X-a)(X+b)$, we can see that the $2$-torsion  points of $E$ are defined over $\Q$, hecne $\Q(E[2]) = \Q$.

For assertion (ii), we shall follow the notation of Definition \ref{1-2-def} and prove via Lemma \ref{1-2-lem: basic properties}.
Since $E$ has semi-stable reduction at $p$, we can see that $E_F$ has good [if $p\nmid N$] or multiplicative [if $p\mid N$] reduction at $w$, and for $p\in\{2,3,l\}$, we have $(p-1) \mid e_p(E,F,l)$, cf.  Lemma \ref{1-2-lem: basic properties}.

In the case of $p\mid N$, since $\Q(\sqrt{-1}) / \Q$ only ramifies at 2, by Lemma \ref{1-2-lem: basic properties}, 
we have $e_2(E,F,l) \mid 2 e_2(3l) \mid 6l$, 
$e_3(E,F,l) = e_3(3l) = e_3(l) \in \{2,6,8\}$, $e_l(E,F,l) = e_l(3l) = e_l(l) \in \{l-1, l(l-1), l^2-1\}$, and for $p\neq 2,3,l$, $e_p(E,F,l) = e_p(3l) = 1$.

Other cases in assertions (ii) and (iii) can be proved using a similar approach.
\end{proof}

\begin{cor} \label{1-2-cor2}
Let $(a,b,c)$ be a triple of non-zero coprime integers such that $a^2 + b^3 = c$.
Let $E$ be the elliptic curve defined over $\Q$ by the equation 
$$Y^2 = X^3 + 3bX + 2a,$$
Let $l\ge 5$ be a prime number, and we shall follow the notation of Definition \ref{1-2-def}.

(i) $E$ has semi-stable reduction at each $p\neq 2,3$;
the $j$-invariant of $E$ is $j(E) = \frac{1728b^3}{c}$,
whose denominator $N$ equals $|c| / \gcd(1728, c)$.

(ii) Suppose that $q\ge 5$ is a prime number, such that $q\neq l$ and $|c|^{1/q} \in \Z$.
Write $F \defeq \Q(\sqrt{-1}, E[2q])$,  $E_F \defeq E\times_{\Q} F$, and let $w$ be a place of $F$ with residue characteristic $p$. 
Then if $p \nmid N$, $E_F$ has good reduction at $w$; if $p \mid N$, $E_F$ has multiplicative reduction at $w$; we have $2\mid e_2(E,F,l)$, $(q-1)\mid e_q(E,F,l)$ and $(l-1) \mid e_l(E,F,l)$.

In the case of $p \nmid N$, we have 
$e_2(E,F,l) \mid 2^5\cdot 3^2$ for $p=2$, $e_q(E,F,l) \in \{q-1, q(q-1), q^2-1 \}$ for $p=q$, $e_l(E,F,l) \in \{l-1, l(l-1), l^2-1\}$ for $p=l$, and $e_p(E,F,l) = 1$ when $p\neq 2, q, l$.

In the case of $p \mid N$, we have
$e_2(E,F,l) \mid 8ql$ for $p=2$, $e_q(E,F,l) \mid 2ql(q-1)$ for $p=q$, $e_l(E,F,l) \mid 2 l (l-1)$ for $p=l$, and $e_p(E,F,l) \mid 2ql$ for $p\neq 2, q, l$.
\end{cor}
\begin{proof}
For assertion (i), by using Tate's algorithm, we have 
\small\begin{align*}
c_4(E) = -144 b, \quad
\Delta(E) = -1728c, \quad 
j(E) = \frac{1728b^3}{c}.
\end{align*}\normalsize
Since $\gcd(a,b,c) = 1$, we have $\gcd(c_4(E), \Delta) = \gcd(1728, c)$, hence $E$ has semi-stable reduction at each $p\neq 2,3$, and $N = |c|/\gcd(1728, c)$.

For assertion (ii), the proof follows a similar approach to those of Corollary \ref{1-2-cor0} and Corollary \ref{1-2-cor1}.
\end{proof}

\begin{cor} \label{1-2-cor3}
Let $(a,b,c)$ be a triple of non-zero coprime integers such that $a + b = c^3$.
Let $E$ be the elliptic curve defined over $\Q$ by the equation 
$$Y^2 + 3cXY + aY= X^3,$$
Let $l\ge 5$ be a prime number, and we shall follow the notation of Definition \ref{1-2-def}.

(i) $E$ has semi-stable reduction at each $p\neq 3$;
the $j$-invariant of $E$ is $j(E) = \frac{27 c^3 (a+9b)^3}{a^3 b}$,
whose denominator $N$ equals $|a^3 b| / \gcd(27, a^3 b)$.

(ii) Write $F\defeq \Q(E[4])$, $E_F \defeq E\times_{\Q} F$, and let $w$ be a place of $F$ with residue characteristic $p$. 
Then if $p \nmid N$, $E_F$ has good reduction at $w$; if $p \mid N$, $E_F$ has multiplicative reduction at $w$; we have $\sqrt{-1} \in F$, $2\mid e_2(E,F,l)$ and $(l-1) \mid e_l(E,F,l)$.

In the case of $p \nmid N$, we have 
$e_2(E,F,l) \mid 96$ for $p=2$, $e_3(E,F,l) \mid 12$ for $p=3$, 
$e_l(E,F,l) \in \{l-1, l(l-1), l^2-1\}$ for $p=l$, and $e_p(E,F,l) = 1$ for $p\neq 2,3,l$.
 
In the case of $p \mid N$, we have
$e_2(E,F,l) \mid 8l$ for $p=2$, $e_3(E,F,l) \mid 8l$ for $p=3$, 
$e_l(E,F,l) \mid 4 l (l-1)$ for $p=l$, 
and $e_p(E,F,l) \mid 4 l$ for $p\neq 2,3,l$.
\end{cor}
\begin{proof}
For assertion (i), by using Tate's algorithm, we have 
\small\begin{align*}
c_4(E) = 9 c (a+9b), \quad
\Delta(E) = 27 a^3 b, \quad 
j(E) = \frac{27 c^3 (a+9b)^3}{a^3 b}.
\end{align*}\normalsize
Since $\gcd(a,b,c) = 1$, we have $\gcd(c_4(E), \Delta) = \gcd(27, a^3 b)$, hence $E$ has semi-stable reduction at each $p\neq 3$, and $N = |a^3 b| / \gcd(27, a^3 b)$.

For assertion (ii), the proof follows a similar approach to those of Corollary \ref{1-2-cor0} and Corollary \ref{1-2-cor1}.
\end{proof}

\begin{cor} \label{1-2-cor4}
Let $(a,b,c)$ be a triple of non-zero coprime integers such that $a + b = c^2$.
Let $E$ be the elliptic curve defined over $\Q$ by the equation 
$$Y^2 = X^3 + 2cX^2 + aX,$$
Let $l\ge 5$ be a prime number, and we shall follow the notation of Definition \ref{1-2-def}.

(i) $E$ has semi-stable reduction at each $p\neq 2$;
the $j$-invariant of $E$ is $ j(E) = \frac{64 (a+4b)^3}{a^2 b}$,
whose denominator $N$ equals $|a^2 b| / \gcd(64, a^2 b)$.

(ii) Write $F\defeq \Q(\sqrt{-1}, E[6])$, $E_F \defeq E\times_{\Q} F$, and let $w$ be a place of $F$ with residue characteristic $p$. 
Then if $p \nmid N$, $E_F$ has good reduction at $w$; if $p \mid N$, $E_F$ has multiplicative reduction at $w$; we have $\sqrt{-1} \in F$, $2\mid e_2(E,F,l)$, $2\mid e_3(E,F,l)$ and $(l-1) \mid e_l(E,F,l)$.

In the case of $p \nmid N$, we have 
$e_2(E,F,l) \mid 2^5 \cdot 3^2$ for $p=2$, $e_3(E,F,l) \mid 2^6 \cdot 3^2$ for $p=3$, 
$e_l(E,F,l) \in \{l-1, l(l-1), l^2-1\}$ for $p=l$, and $e_p(E,F,l) = 1$ for $p\neq 2,3,l$.
 
In the case of $p \mid N$, we have
$e_2(E,F,l) \mid 24 l$ for $p=2$, $e_3(E,F,l) \mid 12 l$ for $p=3$, 
$e_l(E,F,l) \mid 6 l (l-1)$ for $p=l$, 
and $e_p(E,F,l) \mid 6 l$ for $p\neq 2,3,l$.
\end{cor}
\begin{proof}
For assertion (i), by using Tate's algorithm, we have 
\small\begin{align*}
c_4(E) = 16 (a+4b), \quad
\Delta(E) = 64 a^2 b, \quad 
j(E) = \frac{64 (a+4b)^3}{a^2 b}.
\end{align*}\normalsize
Since $\gcd(a,b,c) = 1$, we have $\gcd(c_4(E), \Delta) = \gcd(64, a^2 b)$, hence $E$ has semi-stable reduction at each $p\neq 2$, and $N = | a^2 b |/\gcd(64,  a^2 b)$.

For assertion (ii), the proof follows a similar approach to those of Corollary \ref{1-2-cor0} and Corollary \ref{1-2-cor1}.
\end{proof}

\subsection{Applications of local inequalities}
We shall talk about a modified version of the methods of \cite{Diophantine_after_IUT_I}, Lemma 3.1.
We shall first introduce some symbols and the conditions they should satisfy.
All necessary symbols [i.e. the symbols that need to be defined when applying the conclusions in this subsection] are marked in bold.

\begin{definition} \label{1-3-def: basic notation for global inequalities}
Throughout this section, we will use the following notation:

(a) Let $\boldsymbol{N, n_0, u_0, p_N}$ be positive integers, such that $v_p(n_0 \cdot N)\ge u_0$ and $p \ge p_N$ for each $p\mid N$.
For each prime number $l$, write $N_l \defeq \prod_{p:l\mid v_p(N)} p^{v_p(N)}$.

(b) Let $\boldsymbol{S}$ be a finite set of prime numbers $\ge 5$, with cardinality $n \ge 2$; 
let $2 \le \boldsymbol{k} \le n$, $\boldsymbol{n_1(S), n_k(S)}$ be positive integers, $\boldsymbol{S_1}$ be a finite set of prime numbers;
write $p_0$ for the sallest prime number in $S$, and write $k(S)$ for the product of the $k$ smallest prime numbers in $S$. 

Suppose that for each prime $p\in B\setminus S_1$ [for the definition of $B$, cf. (e)], we have $v_p(N) \ge n_1(S)$;
 for any $k$ distinct prime numbers $l_1, ..., l_k \in S$ and each prime $p$, if $l_i \mid v_p(N)$ for $1\le i\le k$, then $v_p(N) \ge n_k(S)$.

(c) Let $\boldsymbol{\lambda} > 0$ be a real number. For each $l\in S$, suppose that we have real numbers $\boldsymbol{\textbf{Vol}(l)} \ge 0$, $\boldsymbol{a_1(l)} \ge \boldsymbol{a_4(l)} > 0$, such that 
\small\begin{align} \label{1-1-eq1}
\frac{1}{\lambda}\log(l^{-v_l(N)} \cdot N/N_l  ) \le  a_1(l) \cdot \log\rad(N) - a_4(l) \cdot \log\rad(N_l) + \Vol(l)
\end{align}\normalsize

\textbf{(Optional)} In the present paper, if not otherwise specified, we shall take $\lambda = 6$; instead of defining $a_1(l)$, $a_4(l)$, we shall fix a positive integer $\boldsymbol{e_0}$, and write 
\begin{gather*}
a_1(l) \defeq \frac{l^2+5l}{l^2+l-12}\cdot  (1-\frac{1}{e_0  l}), \;
a_4(l) \defeq \frac{l^2+5l}{l^2+l-12}\cdot \frac{1}{e_0}(1- \frac{1}{l}) .
\end{gather*}

(d) \textbf{(Optional)} Let $\boldsymbol{u'_0}\ge u_0$ be a positive integer, $\boldsymbol{P}$ be a subset of prime numbers, such that for each $p\in P$, if $p\mid N$, then $v_p(n_0N)\ge u'_0$.
Let $\boldsymbol{n'_1(S)} \ge n_1(S)$ be a positive integer, such that for each $p\in B' \setminus S_1$ [for the definition of $B'$, cf. (e)], we have $v_p(N') = v_p(N) \ge n'_1(S)$.

(e) We shall write
\small\begin{gather*}
A\defeq \{p\in S : p\mid N\}, \;
B\defeq \{p\mid N: \exists l\in S,\, l\mid v_p(N) \}, \;
C\defeq \{p\mid N: p \notin (A\cup B)\}, \\
A' \defeq P \cap A, \;
B' \defeq P \cap B, \;
C' \defeq P \cap C, \;
N'\defeq \prod_{p\in P}p^{v_p(N)}, \;
n'_0 \defeq \prod_{p\in P}p^{v_p(n_0)}; \;
N_A \defeq \prod_{p\in A}p^{v_p(N)}, \\
N_B \defeq \prod_{p\in B}p^{v_p(N)},  \;
N_C \defeq \prod_{p\in C}p^{v_p(N)}, \;
N'_A \defeq \prod_{p\in A'}p^{v_p(N')}, \; 
N'_B \defeq \prod_{p\in B'}p^{v_p(N')},  \;
N'_C \defeq \prod_{p\in C'}p^{v_p(N')}; \\
a_1 \defeq \frac{\lambda}{n}\sum_{l\in S}a_1(l), \;
a_2 \defeq \frac{\lambda}{n}\sum_{l\in S}\Vol(l), \;
a_3 \defeq \sum_{l\in S}\log(l), \;
a_4 =\lambda\cdot \min_{l\in S}\{ a_4(l) \}, \;
a_5 = \sum_{p\in S_1} \log(p); \\
b_1 \defeq \max\{\frac{a_1}{u_0}, \frac{k}{n} + \frac{a_1-a_4}{n_1(S)} \}, \;
b'_1 \defeq \max\{\frac{a_1}{u'_0}, \frac{k}{n} + \frac{a_1-a_4}{n'_1(S)} \} \le b_1, \\
\quad  b_2 \defeq \frac{a_1}{u_0}\cdot \log(n_0) + a_2 + a_1a_3 + (a_1-a_4)a_5 .
\end{gather*}\normalsize

\end{definition}

\begin{lem} \label{1-3-lem: local-global inequalities}
Proceeding with the notation in Definition \ref{1-3-def: basic notation for global inequalities} and suppose that $\log(N) < n_k(S) \cdot \log(p_N)$.

(i) We have
\small\begin{gather*}
\sum_{l\in S}v_l(N)\cdot\log(l) = \log(N_A),
\; \sum_{p\in B} \log(p) \le  \sum_{l\in S} \log\rad(N_l) ,
\;  \sum_{p\in C} \log(p) \le \frac{\log(n_0 N_C)}{u_0}, 
\\ \sum_{l\in S} \log(N_l)\le (k-1)\cdot\log(N_B),
\; \sum_{p\in B} \log(p) \le \frac{1}{n_1(S)} \cdot \log(N_B) + a_5 ,
\\ \sum_{p\in B} \log(p) \le \frac{1}{n_1(S)} \cdot \log(N_B/N'_B) +  \frac{1}{n'_1(S)} \cdot \log(N'_B) + a'_5 .
\end{gather*}\normalsize

(ii) We have
\small\begin{align}  \label{1-1-eq3}
   \log(N) \le a_1\cdot\log\rad(N) - a_4 \sum_{p\in B}\log(p) + a_2 + \frac{k}{n}(\log(N)-\log(N_C)).
\end{align}\normalsize

(iii) We have
\small\begin{align*}
    \log(N) \le a_1\cdot\log\rad(N_C) + (\frac{k}{n}+ \frac{a_1-a_4}{n_1(S)})(\log(N)-\log(N_C))
    + a_2 + a_1a_3 + (a_1-a_4)a_5 .
\end{align*}\normalsize

(iv) We have 
\small\begin{align*}
\log(N) \le &  a_1 \cdot \log\rad(N_C) 
+ (\frac{k}{n} + \frac{a_1-a_4}{n_1(S)})(\log(N/N')-\log(N_C/N'_C))
\\ & + (\frac{k}{n} + \frac{a_1-a_4}{n'_1(S)})(\log(N')-\log(N'_C))
+ a_2 + a_1a_3 + (a_1-a_4)a'_5 .
\end{align*}\normalsize

\end{lem}
\begin{proof}
For assertion (i), by the definition of the set $A$, we have 
\small\begin{align*}
    \sum_{l\in S}v_l(N)\cdot\log(l) = \sum_{l:l\in S,l\mid N}v_l(N)\cdot\log(l) = \sum_{l\in A}v_l(N)\cdot\log(l) = \log(N_A).
\end{align*}\normalsize

By the definition of $N_l$ and $B$, we have 
\small\begin{align*}
\sum_{l\in S} \log\rad(N_l)
= \sum_{l \in S} \sum_{p: p\mid N, l\mid v_p(N)} \log(p)
= \sum_{p\in B}\log(p)\cdot\big(\sum_{l:l \in S,l\mid v_p(N)} 1 \big)
\ge \sum_{p\in B} \log(p) 
\end{align*}\normalsize

By the definition of $N_l$ and $B$, we also have 
\small\begin{align*}
\sum_{l\in S} \log(N_l)
= \sum_{l \in S} \sum_{p:p\mid N,l\mid v_p(N)} v_p(N)\cdot \log(p)
= \sum_{p\in B}v_p(N)\cdot \log(p)\cdot\big(\sum_{l:l \in S,l\mid v_p(N)} 1 \big).
\end{align*}\normalsize

By the definition of $n_0$ and $u_0$, for each $p\in C$, we have $v_p(n_0 \cdot N) \ge u_0$, hence 
\small\begin{align*}
\sum_{p\in C} \log(p) \le&  \frac{1}{u_0} \sum_{p\in C} (v_p(n_0) + v_p(N) ) \cdot \log(p) 
\\ \le&  \frac{1}{u_0} \sum_{p\in C} v_p(N)\cdot \log(p) 
+ \frac{1}{u_0} \sum_{p} v_p(n_0) \cdot \log(p)   
= \frac{\log(n_0 N_C)}{u_0}.
\end{align*}\normalsize

Now for each prime number $p\in B$, we claim that $\sum_{p:p\in S, p\mid v_l(N)} 1 \le k-1$, which can be proved by contradiction as follows.
Assume that $l_1,\dots, l_k\in S$ are $k$ distinct prime numbers, such that $l_i\mid v_l(N)$ for $1\le i\le k$. Then by the definition of $n_k(S)$, we have $v_p(N) \ge n_k(S)$, hence
\begin{align*}
    \log(N) = \log(N) &\ge v_p(N)\cdot \log(p)\ge n_k(S) \cdot \log(p_N)  > \log(N).
\end{align*}
--- a contradiction. Thus the claim is true, hence we have 
\small\begin{align*}
    \sum_{l\in S} \log(N_l)
    &= \sum_{p\in B} v_p(N)\cdot \log(p)\cdot\big(\sum_{l: l \in S, l\mid v_p(N)} 1 \big)
   \\&\le (k-1)\cdot\sum_{p\in B}v_p(N)\cdot \log(p) = (k-1)\cdot\log(N_B).
\end{align*}\normalsize

By the definition of $n_1(S)$ and $n_1(p)$, we have $\log(p) \le \frac{1}{n_1(S)}\cdot v_p(N)\cdot\log(p)$ for each $p\in B\setminus S_1$, and $\frac{n_1(p)}{n_1(S)} \cdot \log(p) \le \frac{1}{n_1(S)}\cdot v_p(N)\cdot\log(p)$ for each $p\in B\cap S_1$. Hence
\small\begin{align*}
\sum_{p\in B} \log(p) \le & \frac{1}{n_1(S)} \cdot \sum_{p\in B} v_p(N)\cdot\log(p) + \sum_{p\in B\cap S_1}  (1-\frac{n_1(p)}{n_1(S)})\cdot\log(p)
 \\ \le & \frac{1}{n_1(S)} \cdot \log(N_B) + a_5 .
\end{align*}\normalsize
Similarly, we have
\small\begin{align*}
\sum_{p\in B} \log(p)   \le & \frac{1}{n_1(S)} \cdot \log(N_B/N'_B) + \frac{1}{n'_1(S)} \cdot \log(N'_B) + a_5 .
\end{align*}\normalsize
This proves assertion (i).

Next, we consider assertion (ii). 
By  (\ref{1-1-eq1}) we have 
\begin{equation} \small
\begin{aligned} \label{1-1-eq2}
\log(N) \le \lambda a_1(l)\cdot \log\rad(N) - \lambda a_4(l) \cdot \log\rad(N_l)  
 + \lambda \Vol(l)  + \log(N_l) + v_l(N) \cdot \log(l) .
\end{aligned}
\end{equation}
Then by taking the average of (\ref{1-1-eq2}) for $l\in S$ and by assertion (i), we have
\begin{equation*} \small
\begin{aligned} 
\log(N) \le &  a_1 \log\rad(N) - a_4  \sum_{l\in S}\log\rad(N_l)  + a_2 
+ \frac{1}{n}\sum_{l\in S}\log(N_l)  + \frac{1}{n}\sum_{l\in S} v_l(N) \cdot \log(l) 
\\ \le &  a_1 \log\rad(N) - a_4  \sum_{p\in B}\log(p)  + a_2 + \frac{1}{n}(\log(N_A)+(k-1)\log(N_B) ).
\end{aligned}
\end{equation*}
Then since $\log(N_A), \log(N_B) \le \log(N) - \log(N_C)$ by the fact $A\cap C = B \cap C =\emptyset$, we have 
\small\begin{align*}
   \log(N) \le a_1 \log\rad(N) - a_4  \sum_{p\in B}\log(p) + a_2 + \frac{k}{n}(\log(N) - \log(N_C)).
\end{align*}\normalsize
This proves assertion (ii).

For assertion (iii), by the definition of $A$, $B$ and $C$, for each $p\mid N$, we have $p\in A\cup B\cup C \subseteq S\cap B \cap C$, hence by (i) we have (note that $a_1\ge a_4$ by definitinon)
\small\begin{align*}
&a_1 \log\rad(N) - a_4  \sum_{p\in B}\log(p) 
\le a_1\sum_{p\in S}\log(p) + (a_1-a_4)\sum_{p\in B}\log(p) + a_1\sum_{p\in C} \log(p)
\\ \le& a_1 a_3 +  (a_1-a_4)(\frac{\log(N_B)}{n_1(S)}  + a_5) + a_1\sum_{p\in C} \log(p).
\end{align*}\normalsize
Then by (\ref{1-1-eq3})  and  $\log(N_B) \le \log(N) - \log(N_C)$   we have  
\begin{equation*} \small
\begin{aligned}
& \log(N) \le   a_1 \log\rad(N) - a_4  \sum_{p\in B}\log(p)  + a_2 + \frac{k}{n}(\log(N) - \log(N_C) )
\\ \le & a_1 a_3 + (a_1-a_4)(\frac{\log(N_B)}{n_1(S)}  + a_5) + a_1\sum_{p\in C} \log(p) + a_2 + \frac{k}{n}(\log(N) - \log(N_C) )
\\ \le& a_1\cdot\log\rad(N_C) + (\frac{k}{n}+ \frac{a_1-a_4}{n_1(S)})(\log(N)-\log(N_C))
+ a_2 + a_1a_3 + (a_1-a_4)a_5 .
\end{aligned}
\end{equation*}
This proves assertion (iii).

The proof of assertion (iv) is similar to the proof of assertion (iii), where we shall make use of $\sum_{p\in B} \log(p) \le \frac{1}{n_1(S)} \cdot \log(N_B/N'_B) +  \frac{1}{n'_1(S)} \cdot \log(N'_B) + a_5$ by (i), and $\log(N_B/N'_B) \le \log(N/N') - \log(N_C/N'_C)$, $\log(N'_B) \le \log(N') - \log(N'_C)$ by the fact $(B\setminus B') \cap (C\setminus C') = B'\cap C' = B \cap C = \emptyset$.
\end{proof}

\begin{prop} \label{1-3-prop: local-global inequalities}
Proceeding with the notation in Definition \ref{1-3-def: basic notation for global inequalities} and suppose that $\log(N) < n_k(S) \cdot \log(p_N)$.

(i) Suppose that $b_1 < 1$, then we have
\begin{align} \label{1-2-eq0}
(1-b_1) \cdot \log(N) \le b_2 .
\end{align}
Hence if $n_k(S)\cdot \log(p_N) > \frac{b_2}{1-b_1}$, then $\log(N)$ cannot belong to the interval $$(\frac{b_2}{1-b_1}, n_k(S)\cdot \log(p_N) ) .$$

(ii) Suppose that $b_1 \le 1$, $b'_1 < 1$, then we have
\begin{align*}
(1-b'_1) \cdot \log(N') \le b_2 .
\end{align*}
Hence if $n_k(S)\cdot \log(p_N) > \frac{b_2}{1-b'_1}$, then $\log(N')$ cannot belong to the interval $$(\frac{b_2}{1-b'_1}, n_k(S)\cdot \log(p_N) ) .$$

\end{prop}
\begin{proof}
For assertion (i), recall that by Lemma \ref{1-3-lem: local-global inequalities}, (i) we have 
\small\begin{align*}
\log\rad(N_C) = \sum_{p\in C} \log(p) \le \frac{1}{u_0}\cdot \log(n_0 N_C) .
\end{align*}\normalsize
Then by Lemma \ref{1-3-lem: local-global inequalities}, (iii) and the definition of $b_1$, $b_2$, we have
\begin{equation*} \small
\begin{aligned}
\log(N) \le&  \frac{a_1}{u_0}\cdot\log(n_0 N_C) + (\frac{k}{n}+ \frac{a_1-a_4}{n_1(S)})(\log(N)-\log(N_C))
+ a_2 + a_1a_3 + (a_1-a_4)a_5 
\\ =& \frac{a_1}{u_0}\cdot \log(N_C) + (\frac{k}{n}+ \frac{a_1-a_4}{n_1(S)})(\log(N)-\log(N_C)) + b_2 
\\ \le& b_1\cdot \log(N) + b_2 .
\end{aligned}
\end{equation*}
Hence we must have (\ref{1-2-eq0}), and then assertion (i) follows.

For assertion (ii),  similar to the proof of assertion (i), by Lemma \ref{1-3-lem: local-global inequalities}, (i) and (iv) we have
\small\begin{align*}
\log(N) \le &  a_1 \cdot \log\rad(N_C/N'_C) - \frac{a_1}{u_0}\log(n_0/n'_0) 
+ (\frac{k}{n} + \frac{a_1-a_4}{n_1(S)})(\log(N/N')-\log(N_C/N'_C))
\\ & +  a_1 \cdot \log\rad(N'_C) - \frac{a_1}{u'_0}\log(n'_0) + (\frac{k}{n} + \frac{a_1-a_4}{n'_1(S)})(\log(N')-\log(N'_C))
+ b_2 
\\ \le &  b_1 \cdot \log(N/N') + b'_1 \cdot \log(N') + b_2 .
\end{align*}\normalsize
Then since $b_1 \le 1$, we must have $\log(N') \le b'_1 \cdot \log(N') + b'_2$, and hence assertion (ii) follows.
\end{proof}

\begin{tiny-remark} \label{1-3-rmk:  local-global inequalities}
(i) Proposition \ref{1-3-prop: local-global inequalities} provides an algorithm to eliminate impossible value intervals of $\log(N)$ by choosing suitable $n_0, N, u_0, p_N$, $S, k, S_1, n_1(S), n_k(S)$, $\lambda$, $e_0$, $\Vol(l)$, $u'_0, P, n'_1(S)$, as defined in Definition \ref{1-3-def: basic notation for global inequalities}.
From a mathematically rigorous point of view, the value of $\Vol(l)$ is obtained by constructing general $\mu_6$- or $2$-torsion initial $\Theta$-data and estimating the log-volume, and we shall prove $\log(N) \le C_0$ for some explicit (but possibly very large) real number $C_0 > 0$ at first.
Then by choosing finite many suitable and explicit $S$, $k$, $\Vol(l)$, etc.,  in Definition \ref{1-3-def: basic notation for global inequalities}, we can apply Proposition \ref{1-3-prop: local-global inequalities} to show that $\log(N)$ can't belong to the interval $(C, C_0)$, where $C>0$ is a  real number much smaller than $C_0$. Hence we can obtain a smaller upper bound $\log(N) \le C$. Smaller upper bound for $\log(N')$ can be obtained similarly.

(ii)  For an example for the procedure described in (i), cf. \cite{Diophantine_after_IUT_I}, Section 4. The step to show $\log(N) \le C_0$ can be proved via estimation, and the step to show $\log(N) \notin (C,C_0)$ can be done by the computation via computer. 
In this paper, when we apply Proposition \ref{1-3-prop: local-global inequalities}, we shall omit the step to show $\log(N) \le C_0$ using analytic number theory.

Instead, we shall only choose finitely many suitable $S$, $k$, $\Vol(l)$, etc., in Definition \ref{1-3-def: basic notation for global inequalities} by constructing suitable $\mu_6$- or $2$-torsion initial $\Theta$-data to define $\Vol(l)$, and compute the step to show $\log(N) \notin (C,C_1)$ [with $C$ small and $C_1$ bigger] via computer. 

We claim that it is always possible to show $\log(N) < C_0$ for some $C_0 > C_1$ by the procedure described in (i), 
and show $\log(N) \notin (C_1, C_0)$ by computing more special cases via Proposition \ref{1-3-prop: local-global inequalities}. Hence we can deduce that $\log(N) \le C$.
\end{tiny-remark}

\section{The general signatures}
Let $r,s,t \ge 4$ be positive integers. In this section, we shall consider about the generalized Fermat equation $x^r+y^s=z^t$ with signature $(r,s,t)$. 

\subsection{Upper bounds}

\begin{lem}  \label{2-1-lem-1}
Let $r,s,t\ge 4$ be positive integers, $(x,y,z)$ be a triple of positive coprime integers such that $\delta_r x^r + \delta_s y^s = z^t$, where $\delta_r,\delta_s \in \{\pm 1\}$. Let $l\ge 11$ be a prime number.

(i) Let $(a,b,c)$ be a permutation of $(\delta_r x^r, \delta_s y^s, -z^t)$, such that $4\mid (a+1)$ and $16\mid b$. Then $a+b+c = 0$, $\gcd(a,b,c) = 1$.

Let $E$ be the elliptic curve defined over $\Q$ by the equation 
$$Y^2 = X(X- a)(X+ b),$$
then the $j$-invariant
$$j(E) = \frac{(a^2+ab+b^2)^3}{2^{-8}a^2b^2c^2} = \frac{(a^2+ab+b^2)^3}{2^{-8}x^{2r}y^{2s}z^{2t} } \notin \{0,2^6\cdot 3^3,2^2\cdot 73^3\cdot 3^{-4},2^{14}\cdot 31^3\cdot 5^{-3}\} .$$
Write 
$$N \defeq 2^{-8}x^{2r}y^{2s}z^{2t}, \; N_l \defeq N^{\lagl}, \; N'_l \defeq (N/N_l)_{(l)}, \;  N'_{2l} \defeq (N/N_l)_{(2l)}. $$
Then $x,y,z \ge 2$, $N$ is the denominator of $j(E)$ and $E$ is semi-stable.

(ii) Write $F\defeq \Q(\sqrt{-1}, E[3])$, then when $N'_l \neq 1$, by the construction in \cite{Diophantine_after_IUT_I}, Proposition 2.5, there exists a $\mu_6$-initial $\Theta$-data $\mathfrak D = \mathfrak D(E,F,l,\mu_6)$ of type $(l, N, N'_l)$.

Let $\mathfrak{R}_l^{(1)}$ be the ramification dataset [cf. \cite{Diophantine_after_IUT_I}, Definition 1.12] consists of the following data:
\begin{gather*}
l_0 = l, \; e_0 = 3, \; S_0=\{2,3,l\}, \; S_{\gen}^{\multi} = \{1,3,l,3l\}, \; S_2^{\good} = \{2\}, \\ 
S_2^{\multi} = \{2,6,2l,6l\}, \; S_3^{\good} = \{2,6,8\},  \; S_3^{\multi} = \{2,6,2l,6l\}, \\
S_l^{\good} = \{l-1, l(l-1), l^2-1\}, \;   S_l^{\multi} = \{l-1, 3(l-1), l(l-1), 3l(l-1)\} .
\end{gather*}
Let $\mathfrak{R}_l^{(2)}$ be the ramification dataset obtained by setting $S_l^{\good} = \emptyset$ in the definition of $\mathfrak{R}_l^{(1)}$, i.e. the ramification dataset consists of the following data:
\begin{gather*}
l_0 = l, \; e_0 = 3, \; S_0=\{2,3,l\}, \; S_{\gen}^{\multi} = \{1,3,l,3l\}, \; S_2^{\good} = \{2\}, \\ 
S_2^{\multi} = \{2,6,2l,6l\}, \; S_3^{\good} = \{2,6,8\},  \; S_3^{\multi} = \{2,6,2l,6l\}, \\
S_l^{\good} = \emptyset, \;   S_l^{\multi} = \{l-1, 3(l-1), l(l-1), 3l(l-1)\} .
\end{gather*}
Then $\mathfrak D$ admits $\mathfrak R_l^{(1)}$; and if $l\mid xyz$, then $\mathfrak D$ admits $\mathfrak R_l^{(2)}$.
Hence we have 
\begin{equation} \small \label{2-1-eq1}
    \frac{1}{6}\log(N'_l) \le  \frac{l^2+5l}{l^2+l-12}\cdot \big( (1-\frac{1}{3  l})\cdot \log\rad(N) - \frac{1}{3}(1- \frac{1}{l}) \log\rad(N_l)  \big) + \Vol(l),
\end{equation}
where we write $\Vol(l) \defeq \Vol(\mathfrak R_l^{(1)})$ if $l\nmid xyz$, and write $\Vol(l) \defeq \Vol(\mathfrak R_l^{(2)}) \le \Vol(\mathfrak R_l^{(1)})$ if $l\mid xyz$.
The above inequality holds for $N'_l = 1$.

(iii) Write $F\defeq \Q(\sqrt{-1})$, then when $N'_{2l} \neq 1$, by the construction in Proposition \ref{1-prop: construction of mu_2 initial Theta-data}, there exists a $2$-torsion initial $\Theta$-data $\mathfrak D = \mathfrak D(E,F,l,2\text{-tor})$ of type $(l, N, N'_{2l})$.

Let $\mathfrak{R}_l^{(3)}$ be the ramification dataset [cf. \cite{Diophantine_after_IUT_I}, Definition 1.12] consists of the following data:
\begin{gather*}
l_0 = l, \;  e_0 = 1, \; S_0=\{2,l\}, \; S_{\gen}^{\multi} = \{1,l\}, \; S_2^{\good} = \{2\}, \; 
S_2^{\multi} = \{2,2l\}, \\
S_l^{\good} = \{l-1, l(l-1), l^2-1\}, \;   S_l^{\multi} = \{l-1, l(l-1)\} .
\end{gather*}
Let $\mathfrak{R}_l^{(4)}$ be the ramification dataset obtained by setting $S_l^{\good} = \emptyset$ in the definition of $\mathfrak{R}_l^{(3)}$.
Then $\mathfrak D$ admits $\mathfrak R_l^{(3)}$; and if $l\mid xyz$, then $\mathfrak D$ admits $\mathfrak R_l^{(4)}$.
Hence we have 
\begin{equation} \small \label{2-1-eq2}
    \frac{1}{6}\log(N'_l) \le  \frac{l^2+5l}{l^2+l-12}\cdot (1-\frac{1}{l})\cdot \log\rad(N/N_l)  + \Vol(l) ,
\end{equation}
where we write $\Vol(l) \defeq \Vol(\mathfrak R_l^{(3)})$ if $l\nmid xyz$, and write $\Vol(l) \defeq \Vol(\mathfrak R_l^{(4)}) \le \Vol(\mathfrak R_l^{(3)})$ if $l\mid xyz$.

\end{lem}
\begin{proof}
For assertion (i), since $r,s,t \ge 4$, we have $v_2(x^r y^s z^t) \ge 4$, hence $(a,b,c)$ exists. Then by Corollary \ref{1-2-cor1}, we have $j(E) = \frac{(a^2+ab+b^2)^3}{2^{-8}a^2b^2c^2}$, $N$ is the denominator of $j(E)$, and $E$ is semi-stable. In addition, we have $x,y,z \ge 2$ by Catalan's conjecture.

For assertion (ii), since $j(E) \notin J$, $E$ is semi-stable over $\Q$, $N'_l \neq 1$, $E_F$ is semi-stable [cf. Corollary \ref{1-2-cor4}], $l\ge 11$, $F = \Q(\sqrt{-1}, E[12])$ is Galois over $\Q$ [cf. \cite{Diophantine_after_IUT_I}, Proposition 2.1],
by the construction in \cite{Diophantine_after_IUT_I}, Proposition 2.5, there exists a $\mu_6$-initial $\Theta$-data $\mathfrak D = \mathfrak D(E,F,l,\mu_6)$ of type $(l, N, N'_l)$. Then by Corollary \ref{1-2-cor0}, we can see that $\mathfrak D$ admits $\mathfrak R_l^{(1)}$, and $\mathfrak D$ admits $\mathfrak R_l^{(2)}$ when $l\mid N$ [which is equivlent to $l\mid xyz$], hence by Corollary \ref{1-cor: The log-volume of ramification dataset} we have (\ref{4-1-eq1}). Since the log-volume of a ramification dataset is non-negative, we have $\Vol(l) \ge 0$, hence  (\ref{4-1-eq1}) holds for $N'_l = 1$.

For assertion (iii), by using Corollary \ref{1-2-cor1}, (iii), and using Proposition \ref{1-prop: construction of mu_2 initial Theta-data} instead of \cite{Diophantine_after_IUT_I}, Proposition 2.5,
assertion (iii) can be proved similarly to the proof of assertion (ii).

\end{proof}

\begin{lem}  \label{2-1-lem-2}
Let $r,s,t \ge 4$ be positive integers, $(x,y,z)$ be a triple of positive coprime integers such that $\delta_r x^r + \delta_s y^s =  z^t$, where $\delta_r,\delta_s \in \{\pm 1\}$.

Let $S$ be a finite set of prime numbers with cardinality $n = |S| \ge 2$, such that for each $l\in S$, we have $l \ge 11$;
let $2\le k\le S$ be a positive integer, $p_0$ be the smallest prime number in $S$, $k(S)$ be the product of $k$ smallest prime numbers in $S$;
let $n_0 = 2^8$, $N = n_0^{-1}x^{2r}y^{2s}z^{2t}$, $u_0 = 8$, $p_N = 2$, $e_0 = 3$,  $S_1 = \{2\}$, $n_1(S) = 2 p_0$, $n_k(S) = 2 k(S)$;
for each $l\in S$, write $\Vol(l) \defeq \Vol(\mathfrak R_l^{(1)})$, cf. Lemma \ref{2-1-lem-1}, (ii).
Suppose that one of the following situations is satisfied:

\begin{itemize}
\item [(a)] Let $u'_0 \ge u_0$ be a positive integer, $P\subseteq \{p: v_p(N) \ge u'_0 \} \cup \{p: p\nmid N\}$, $n'_1(S) = \max\{u'_0, 2p_0\}$.
In particular, we can let $u'_0 = u_0$, $P = \{p: p\mid N\}$, $n'_0 = n_0$ and $N' = N$.

\item [(b)] Let $u'_0 \ge u_0$ be a positive integer, $P\subseteq \{p: v_p(N) \ge u'_0 \} \cup \{p: p\nmid N\}$, $n'_1(S) = \min\{v_p(N): p\in B'\setminus S_1\}$ [for the definition of $B'$, cf. Definition \ref{1-2-def}, (e)], $n'_0 = 2^8$ if $2\in P$ and $n'_0 = 1$ if $2\notin P$.

\item [(c)] In the situation of (b), suppose that $s\ge r\ge t$. 
If $l\nmid rst$ for each $l\in S$, then we can let $u'_0 = 2t$, $n'_1(S) = u'_0 p_0$, $P=\{p: p\mid xyz\}$, $N' = N$;
if $l\nmid rs$ for each $l\in S$, then we can let $u'_0 = 2r$, $n'_1(S) = u'_0 p_0$, $P=\{p: p\mid xy\}$, $N' = (n'_0)^{-1} x^{2r} y^{2s}$;
if $l\nmid s$ for each $l\in S$, then we can let $u'_0 = 2s$, $n'_1(S) = u'_0 p_0$, $P=\{p: p\mid x\}$, $N' = (n'_0)^{-1} y^{2s}$.

\end{itemize}

Then the assumptions in Definition \ref{1-3-def: basic notation for global inequalities} are satisfied, and we can define $b_1$, $b_2$, $b'_1$. 
Hence if $b_1 \le  1$, $b'_1 < 1$ and $n_k(S)\cdot \log(2) > \frac{b_2}{1-b'_1}$, then by Lemma \ref{1-3-lem: local-global inequalities}, $\log(N')$ cannot belong to the interval $(\frac{b_2}{1-b'_1}, n_k(S)\cdot \log(2) ) .$

\end{lem}
\begin{proof}
To show that the assumptions in Definition \ref{1-3-def: basic notation for global inequalities} are satisfied, we only need to check for the definitions of $u_0, n_1(S), n_k(S), u'_0, n'_1(S)$. 
The check for $u_0, u'_0$ is trivial, and we shall only check for $n_1(S)$, and check for $n'_1(S)$ in the second case of the situation of (c) as an example.

For each $p\in B$, there exists $l\in S$, such that $l\mid v_p(N)$. Since $l\ge 11$, and we have $2\mid v_p(N)$ by the definition of $N$, we have $2l \mid v_p(N)$, hence $v_p(N) \ge 2l \ge 2p_0 = n_1(S)$. Thus $n_1(S)$ is well-defined.

In the second case of the situation of (c), for each $p\in B'$, $p\neq 2$, there exists $l\in S$, such that $l\mid v_p(N)$. 
Since $l\ge 11$, $l\nmid rs$, and $v_p(N)$ is a multiple of $2r$, $2s$ by the definition $P = \{p: p\mid xy\}$ and $N' = (n'_0)^{-1} x^{2r} y^{2s}$, we can see that $v_p(N)$ is a multiple of $2rl$ or $2sl$, hence $v_p(N) \ge 2s \cdot l \ge u'_0 p_0 = n'_1(S)$. Thus $n'_1(S)$ is well-defined.
\end{proof}

\begin{lem} \label{2-1-lem-3}
Let $r,s,t \ge 4$ be positive integers, $(x,y,z)$ be a triple of positive coprime integers such that $\delta_r x^r + \delta_s y^s =  z^t$, where $\delta_r,\delta_s \in \{\pm 1\}$.

Let $S$ be a finite set of prime numbers with cardinality $n = |S| \ge 2$, such that for each $l\in S$, we have $l \ge 11$ and $l\nmid rst$;
let $2\le k\le S$ be a positive integer, $p_0$ be the smallest prime number in $S$, $k(S)$ be the product of $k$ smallest prime numbers in $S$;

let $n_0 = 1$, $N \defeq (x^{2r} y^{2s} z^{2t})_{(2)}$, $8 \le u_0 = 8 \le \min\{2r, 2s, 2t\} $, $p_N = 3$, $e_0 = 1$, $S_1 = \emptyset$, $n_1(S) = u_0 \cdot p_0$, $n_k(S) = u_0 \cdot k(s)$;

Suppose that one of the following situations is satisfied:

\begin{itemize}
\item [(a)] Let $u'_0 \ge u_0$ be a positive integer, $P\subseteq \{p: v_p(N) \ge u'_0 \} \cup \{p: p\nmid N\}$, $n'_1(S) = \max\{u'_0, 2p_0\}$.
In particular, we can let $u'_0 = \min\{3r, s\}$, $P = \{p: p\mid N\}$, $n'_0 = n_0$ and $N' = N$.

\item [(b)] Let $u'_0 \ge u_0$ be a positive integer, $P\subseteq \{p: v_p(N) \ge u'_0 \} \cup \{p: p\nmid N\}$, $n'_1(S) = \min\{v_p(N): p\in B'\setminus S_1\}$ [for the definition of $B'$, cf. Definition \ref{1-2-def}, (e)], $n'_0 = 27$ if $3\in P$ and $n'_0 = 1$ if $2\notin P$.

\item [(c)] In the situation of (b), suppose that $s\ge r\ge t$. 
Then we can let $u'_0 = 2t$, $P = \{p: p\mid N\}$, $N' = N$; 
or let $u'_0 = 2r$, $P = \{p: p\mid xy\}$, $n'_0 = \gcd(2^8, x^{2r} y^{2s})$ and $N' = (n'_0)^{-1} x^{2r} y^{2s}$; 
or let $u'_0 = 2s$, $P = \{p: p\mid y\}$, $n'_0 = \gcd(2^8, y^{2s})$ and $N' \defeq (n'_0)^{-1} x^{2r}$.
\end{itemize}

Then the assumptions in Definition \ref{1-3-def: basic notation for global inequalities} are satisfied, and we can define $b_1$, $b_2$, $b'_1$. 
Hence if $b_1 \le  1$, $b'_1 < 1$ and $n_k(S)\cdot \log(3) > \frac{b_2}{1-b'_1}$, then by Lemma \ref{1-3-lem: local-global inequalities}, $\log(N')$ cannot belong to the interval $(\frac{b_2}{1-b'_1}, n_k(S)\cdot \log(3) ) .$

\end{lem}
\begin{proof}
The proof follows from Lemma \ref{2-1-lem-1} and is elementary, which is similar to the proof of Lemma \ref{2-1-lem-2}.
\end{proof}

\begin{prop} \label{2-1-prop-bound}
Let $r,s,t \ge 4$ be positive integers, $(x,y,z)$ be a triple of positive coprime integers such that $\delta_r x^r + \delta_s y^s =  z^t$, where $\delta_r,\delta_s \in \{\pm 1\}$.
Assume without loss of generality that $s\ge r\ge t\ge 4$.

(i) Let $s_0 \ge 4$ be a positive integer. 
Suppose that $s_0 \le t$, let $h = \log(x^{r}y^{s}z^{t})$; 
or $s_0 \le r$, let $h = \log(x^{r}y^{s})$; 
or $s_0 \le s$, let $h = \log(y^{s})$;
or for a fixed prime number $p$, let $s_0 = v_p(N)$, $h =  \frac{1}{2} v_p(N) \cdot \log(p)$.

Then we have $h\le 1693$ for $s_0 = 4$;
$h < 864$ for $s_0 = 5$;
$h < 668$ for $s_0 = 6$;
$h < 571$ for $s_0 \ge 7$;
and $h < 427.5$ for $s_0\ge 600$.
Hence $s = \max\{r,s,t\} \le 616$.

(ii) Let $s_0 \ge 4$ be a positive integer. 
Suppose that $s_0 \le t$, let $h = \log( (x^{r}y^{s}z^{t})_{(2)} )$; 
or $s_0 \le r$, let $h = \log( (x^{r}y^{s})_{(2)} )$; 
or $s_0 \le s$, let $h = \log( (y^{s})_{(2)} )$;
or for a fixed prime number $p\neq 2$, let $s_0 = v_p(N)$, $h =  v_p(x^r y^s z^t) \cdot \log(p)$.

Then we have $h\le 807$ for $s_0 = 4$;
$h < 395$ for $s_0 = 5$;
$h < 312$ for $s_0 = 6$;
$h < 268$ for $s_0 \ge 7$.
\end{prop}
\begin{proof}
For assertion (i), the upper bounds of $h$ follow from Lemma \ref{2-1-lem-2} and the computation in \cite{code_of_zpzhou}, also cf. Remark \ref{1-3-rmk:  local-global inequalities}. 
By Catalan's conjecture, we have $x,y,z \ge 2$.
Note that if $s\ge 600$, then we have $2^s \le h = \log(y^s) \le 427.5$.  
Hence $s \le 427.5 / \log(2) < 617$, thus $s\le 616$.
For assertion (ii), the upper bounds of $h$ follow from Lemma \ref{2-1-lem-3} and the computation in \cite{code_of_zpzhou}.
\end{proof}

\subsection{Structure of solutions}
We shall analyze the structure of the solution $(x,y,z)$ in more detail.

\begin{lem}  \label{2-2-lem-1}
Let $r,s,t \ge 4$ be positive integers, $(x,y,z)$ be a triple of positive coprime integers such that $\delta_r x^r + \delta_s y^s =  z^t$, where $\delta_r,\delta_s \in \{\pm 1\}$.

For each prime number $p\ge 11$, denote [cf. Lemma \ref{2-1-lem-1}]
\small\begin{align*}
a_1(p) &\defeq 3 \cdot \frac{p^2+5p}{p^2+p-12}\cdot (1-\frac{1}{p}),  \\
a_2(p) &\defeq \max\{3 \cdot \Vol(\mathfrak R_p^{(3)})  + a_1(p)\cdot \log(2), 3 \cdot \Vol(\mathfrak R_p^{(4)}) + a_1(p)\cdot \log(2p) \} .
\end{align*}\normalsize
Then $3 < a_1(p) \le 4$, and we have $a_2(11) \le 71$, $a_2(13) \le 74$, $a_2(17) \le 80$, $a_2(19) \le 84$, $a_2(23) \le 91.1$.

Let $l\ge 11$ be a prime number such that $l\nmid r$.
Write uniquely \small\begin{gather}  \label{2-2-eq1}
x = x_1 \cdot 2^{r_2} \cdot l^{r_l} \cdot x_l^{l},
\end{gather}\normalsize
where $x_1, x_l, 2, l$ are pairwise coprime positive integers, $r_2, r_l \ge 0$, $x_l^l = x_{(2l)}^{\lagl}$.

(i) We have 
\small\begin{align}  \label{2-2-eq2}
a_2(l) \ge r \cdot \log(x_1) - a_1(l) \cdot \log\rad(x_1) \ge  (r - 4) \cdot \log(x_1).
\end{align}\normalsize

(ii) we have $x_l \in \{1,3,5\}$ if $(r,l) = (4,11)$;
$x_l \in \{1,3\}$ if $(r,l) = (4,13)$, $(4,17)$, $(5,11)$, $(5,13)$, $(6,11)$; 
and $x_1 = 1$ otherwise.

(iii) We have $r_l \le \frac{37}{r}$. Hence if $r\ge 38$, then we have $r_l = 0$, $x_l = 1$, $x = x_1 \cdot 2^{r_2}$; if $r\ge 69$, then we further have $x_1 = 1$, $x = 2^{r_2}$.

(iv) Suppose that $l\nmid s$, write uniquely $y = y_1 \cdot 2^{s_2} \cdot l^{s_l} \cdot y_l^{l}$, where $y_1, y_l, 2, l$ are pairwise coprime, $s_2, s_l \ge 0$, $y_l^l = y_{(2l)}^{\lagl}$. Then we have $x_l = 1$ or $y_l = 1$.
\end{lem}

\begin{proof}
For assertion (i), note that we have $\log\rad(N / N^{\lagl}) \le \log\rad( N_{(2l)} / N_{(2l)}^{\lagl} ) + \log(2)$ when $l\nmid xyz$, and $\log\rad(N/ N^{\lagl}) \le \log\rad( N_{(2l)} / N_{(2l)}^{\lagl} ) + \log(2l)$ when $l\mid xyz$.
Hence by (\ref{2-1-eq2}) we have 
\small\begin{align}   
\frac{1}{2}\log(N_{(2l)} / N_{(2l)}^{\lagl}) \le a_1(l)\cdot \log\rad( N_{(2l)} / N_{(2l)}^{\lagl} ) + a_2(l) .
\end{align}\normalsize
Since $l\nmid r$, by (\ref{2-2-eq1}) we have $x_1^r \mid N_{(2l)} / N_{(2l)}^{\lagl} $.
Then since $r,s,t \ge 4 \ge a_1(l)$, we have
\small\begin{align*} 
a_2(l) 
&\ge\frac{1}{2}\log(N_{(2l)} / N_{(2l)}^{\lagl}) - a_1(l)\cdot \log\rad( N_{(2l)} / N_{(2l)}^{\lagl} ) \\
&\ge \log(x_1^r)-a_1(l)\cdot \log\rad(x_1) + \sum_{p: p\mid yz, l\nmid v_p(y^s z^t)} (v_p(y^s z^t) -4)\cdot\log(p) \\
& \ge r \cdot \log(x_1) - a_1(l) \cdot \log\rad(x_1) .
\end{align*}\normalsize
This proves assertion (i).

For assertion (ii), let $q$ be the smallest prime number $q\ge 11$, such that $q\neq l$, $q\nmid r$ and $q \nmid x_l$.
Write uniquely $x = x'_1 \cdot 2^{r_2} \cdot q^{r_q} \cdot x_q^{q}$ similar to (\ref{2-2-eq1}), where $x'_1, x_q, 2, q$ are pairwise coprime, $r_2, r_q \ge 0$, $x_q^q = x_{(2q)}^{\lagk{q}}$.
Then by  (\ref{2-2-eq2}) we have $(r - 4) \cdot \log\rad(x'_1) \le a_2(q)$.

Note that we have $\gcd(x_l, 2q) = 1$, and we claim that $\gcd(x_l, x_q) = 1$. 
In fact, suppose that we have $p \mid x_l$, $p\mid x_q$ for some prime $p$, then $p\ge 3$ and $ql \mid v_p(x)$ by definition. 
Hence by Proposition \ref{2-1-prop-bound}, we can take $s_0 \defeq v_p(x) \ge rql > 100$, then $268 > h\defeq v_p(x^r) \cdot \log(p) > r s_0 \cdot \log(3) > 4\cdot 100 \cdot \log(3) > 400$ ---  a contradiction! 
Hence we have $\gcd(x_l, x_q) = 1$.

Since $x_l^l \mid x = x'_1 \cdot 2^{r_2} \cdot q^{r_q} \cdot x_q^{q}$ and $\gcd(x_l, 2 q x_q) = 1$, we can deduce that $x_l^l \mid x'_1$. 
Then by (i), $a_2(q) \ge r\cdot \log(x'_1) - 4\log\rad(x'_1)  \ge (rl-4) \log(x_l)$, thus 
\small\begin{align}  \label{2-2-eq3}
\log(x_l) \le \frac{a_2(q)}{rl-4} .
\end{align}\normalsize

By Proposition \ref{2-1-prop-bound} we have 
$268 \ge \log(x_l^{l r}) = lr \cdot \log(x_l) \ge 11 r \cdot \log(x_l)$,
hence $\log(x_l) \le \frac{24.5}{r}$, $x_l \le e^{24.5 / r}$. 
Consider in the following cases.

(1) In the case of $r\ge 8$, we have $x_l \le  e^{25/r} \le e^{25 / 8} < 30$. Then since $r \le 616$ by Proposition \ref{2-1-prop-bound}, at most four of $11,13,17,19,23$ divide $l$, $r$ or $x_l$. Hence we can deduce that $q\le 23$, and $\log(x_l) \le 92 / (8\cdot 11 - 4) < \log(3)$ by \ref{2-2-eq3}, thus $x_1 < 3$. Then since $x_1$ is coprime to $2$, we have $x_1 = 1$.

(2) In the case of $r \in \{5,6,7\}$, since $x_l \le e^{24.5/r} \le e^{24.5 / 5} < 140 < 11\cdot 13$, we can also deduce that $q\le 17$. Hence by (\ref{2-2-eq3}), we have $\log(x_l) \le 80 / (5 \cdot 11 - 4) < \log(5)$, thus $x_1 < 5$.
Then we can also deduce that $q\le 13$.

When $l = 11$, we can take $q = 13$, and we have $x_l < 74 /  (7 \cdot 11 - 4) < \log(3)$, $x_l = 1$ if $r = 7$; $x_l < 74 /  (5 \cdot 11 - 4) < \log(5)$, $x_l\in\{1,3\}$ if $r \in \{5,6\}$. Similarly, if $l \ge 13$, we can take $q = 11$, and we have $x_l=1$ if $r=6$ or $l\ge 17$; $x_l \in \{1,3\}$ if $r=5$ and $l=13$.

(3) In the case of $r=4$, similar to the proof of case (2), we can first show that $x_l < 10$. Then we can deduce that $q \in \{ 11, 13\}$.
When $l = 11$, we can take $q=13$, hence $\log(x_l) \le \frac{74}{rl-4} < \log(7)$, $x_l \in \{1,3,5\}$. Similarly, if $l \ge 13$, we can take $q = 11$, and we have $x_l=1$ if $l\ge 19$; $x_l \in \{1,3\}$ if $l\in \{13,17\}$.

In conclusion, we have $x_l \in \{1,3,5\}$ if $(r,l) = (4,11)$;
$x_l \in \{1,3\}$ if $(r,l) = (4,13)$, $(4,17)$, $(5,11)$, $(5,13)$, $(6,11)$; 
and $x_1 = 1$ otherwise.

For assertion (iii), let $q$ be the smallest prime number $q\ge 11$, such that $q\neq l$, $q\nmid r$ and $q \nmid r_l$. 
Write uniquely $x = x'_1 \cdot 2^{r_2} \cdot q^{r_q} \cdot x_q^{q}$ similar to (\ref{2-2-eq2}), where $x'_1, x_q, 2, q$ are pairwise coprime, $r_2, r_q \ge 0$, $x_q^q = x_{(2q)}^{\lagk{q}}$.
Then by  (\ref{2-2-eq2}) we have $(rr_l -4)\cdot \log(l) \le r\cdot \log(x'_1) - 4\log\rad(x'_1) \le a_2(q)$, thus 
\small\begin{align}  \label{2-2-eq4}
r_l \le \frac{a_2(q)/\log(l)+4}{r} .
\end{align}\normalsize

Similar to the proof of assertion (ii), when $r\ge 100$, we can show that $q\le 23$, $a_2(q) \le 91.1$, hence by (\ref{2-2-eq4}) we have $r_l < 1$, $r_l = 0$;
when $r < 100$, we can show that $q\le 19$ hence $r_l \le 10$, then we can show that $q\le 17$, $a_2(q) \le 80$, hence by (\ref{2-2-eq4}) we have $r_l \le \frac{80/\log(l)+4}{r} < \frac{38}{r}$, thus $r_l \le \frac{37}{r}$.
Then if $r\ge 38$, we have $r_l = 0$, $x_l = 1$, hence $x = x_1 \cdot 2^{r_2}$.

Suppose that  $r\ge 69$, then $x = x_1 \cdot 2^{r_2}$, and for each prime number $p\ge 11$, $p\nmid r$, by (\ref{2-2-eq2}) we have $\log(x_1) < a_2(p) / (r - 4)$.
If $11\nmid r$, take $p=11$, then we have $\log(x_1) < a_2(11) / (69 - 4) < \log(3)$, hence $x_1 = 1$; 
if $11\mid r$ and $13\nmid r$, then $r\ge 77$, take $p=13$, then we have $\log(x_1) < a_2(13) / (77 - 4) < \log(3)$, hence $x_1 = 1$; 
if $11,13\mid r$, since $r\le 616$, we have $r > 100$, $17\nmid r$, we can take  $p=17$, then $\log(x_1) < a_2(17) / (100 - 4) < \log(3)$, hence $x_1 = 1$.
In summary, we have $x_1 = 1$, $x = 2^{r_2}$ for $r\ge 69$.

assertion (iv) can be proved by a similar approach to the proof of assertion (i).
\end{proof}

\begin{lem}  \label{2-2-lem-2}
Let $r,s,t \ge 4$ be positive integers, $(x,y,z)$ be a triple of positive coprime integers such that $\delta_r x^r + \delta_s y^s =  z^t$, where $\delta_r,\delta_s \in \{\pm 1\}$.

For each prime number $p\ge 17$, write [cf. Lemma \ref{2-1-lem-1}]
\small\begin{gather*}
a_1(p) \defeq 3 \cdot \frac{p^2+5p}{p^2+p-12}\cdot (1-\frac{1}{3p}),  \\
a_2(p) \defeq \max\{3 \cdot \Vol(\mathfrak R_p^{(1)}) + a_1(p)\cdot \log(2), 3 \cdot \Vol(\mathfrak R_p^{(2)}) + a_1(p)\cdot \log(2p) \} .
\end{gather*}\normalsize
Then $3 < a_1(p) < 4$, and we have $a_2(17) \le 156$, $a_2(19) \le 164$, $a_2(23) \le 182$, $a_2(29) \le 210$, $a_2(31) \le 219$, $a_2(37) \le 248$.

Write $r_2 \defeq v_2(x)$.
Let $l$ be a prime number $\ge 17$, such that $l\nmid rst$, $l\nmid (r_2 r - 4)$. 
Write uniquely \small\begin{gather*}  
x = x_1 \cdot 2^{r_2} \cdot l^{r_l} \cdot x_l^{l}, \;
y = y_1 \cdot 2^{s_2} \cdot l^{s_l} \cdot y_l^{l}, \;
z = z_1 \cdot 2^{t_2} \cdot l^{t_l} \cdot z_l^{l}.
\end{gather*}\normalsize
where $x_1, x_l, y_1, y_l, z_1, z_l, 2, l$ are pairwise coprime positive integers, 
$r_2, r_l, s_2, s_l, t_2, t_l \ge 0$, 
$x_l^l = x_{(2l)}^{\lagl}$, $y_l^l = y_{(2l)}^{\lagl}$, $z_l^l = z_{(2l)}^{\lagl}$.

(i) We have 
\small\begin{align}  \label{2-2-eq5}
v_2(x^r) = r_2 r \le ( a_2(l) + 4\log\rad(x_l y_l z_l) ) / \log(2) + 4 .
\end{align}\normalsize

(ii) We have $v_2(x^r) = r_2r \le 306$, hence $r_2 \le \frac{306}{r}$.
Moreover, we have $r \le 303$.
\end{lem}
\begin{proof}
For assertion (i), since $a_2(l) \ge 0$, the right hand of (\ref{2-2-eq5}) is $\ge 4$, hence (\ref{2-2-eq5}) is valid when $r_2r \le 4$.
Now we shall assume that $r_2r > 4$.
Then since $\gcd(x,y,z) = 1$ and $r_2>0$, we have $s_2 = t_2 = 0$.

When $l\mid xyz$, following the notation in Lemma \ref{2-1-lem-1} and (\ref{2-1-eq1}),
we have $l \nmid v_2(N) = 2(r_2r-4) > 0$, $N_l = (x_l^{2r} y_l^{2s} z_l^{2t})^l$, $N'_l = 2^{-8}x_1^{2r} y_1^{2s} z_1^{2t}$, $\log\rad(N) = \log\rad(2l x_1y_1z_1 x_l y_lz_l)$.
Then by (\ref{2-1-eq1}) we have
\begin{equation*} \small 
    \log(x_1^{r} y_1^{s} z_1^{t}) \le  a_1(l) \cdot \log\rad(N)  + 3\cdot \Vol(\mathfrak R_l^{(2)}) \le a_1(l) \cdot \log\rad(x_1y_1z_1 x_ly_lz_l) + a_2(l) .
\end{equation*}
When $l\nmid xyz$, we have $\log\rad(N) = \log\rad(2 x_1y_1z_1 x_l y_lz_l)$. And we can prove $ \log(x_1^{r} y_1^{s} z_1^{t}) \le a_1(l) \cdot \log\rad(x_1y_1z_1 x_ly_lz_l) + a_2(l)$ similarly.

Then since $a_1(l) < 4$, we have
\small\begin{align*} 
    a_2(l) + 4\log\rad(x_l y_l z_l)  
    &\ge \log(x_1^{r} y_1^{s} z_1^{t}) - a_1(l) \cdot \log\rad(x_1 y_1 z_1) \\
    &\ge \sum_{p} (v_p(x_1^{r} y_1^{s} z_1^{t})-4)\cdot\log(p)
    \ge (r_2r-4) \cdot \log(2) .
\end{align*}\normalsize
This proves assertion (i).

For assertion (ii), recall that at most one of $r,s,t$ equals to $4$. Hence if $l=17$, we have $x_ly_lz_l \in \{1,3\}$; if $l\ge 19$, we have $x_ly_lz_l = 1$.
Note that we have $a_2(17) + 4\log\rad(x_l y_l z_l)  \le a_2(17) + 4\log(3) < 160 < a_2(19)$. 

We shall assume that each of $s,t$ is divided by at most one prime number $p \ge 17$, because if we have $p \mid s$, then we can replace $(y,s)$ by $(y^{s/p},p)$.
Similarly, if $r' \mid r$, $r' \ge 4$, then an upper bound of $v_2(x^r)$ when ``$r=r'$'' is also an upper bound of $v_2(x^r)$.

Since $r\ge 4$, note that by Proposition \ref{2-1-prop-bound}, we have $\log(2^{r_2 r}) \le \log(x^r) \le 1693$, hence $r_2r \le \frac{1693}{\log(2)} <2443$, $r_2r \le 2442$. Moreover, let $q$ be the smallest prime $\ge 17$ such that $q\nmid rst$, $q\nmid (r_2r-4)$, then by (\ref{2-2-eq5}), we have $r_2 r \le (a_2(17)+4\log(3)) / \log(2) + 4$ if $l=17$, and $r_2 r \le a_2(l) / \log(2) + 4$ if $l \ge 19$.
Since there are only finitely many $r,s,t,r_2$, computation in \cite{code_of_zpzhou} shows that we have $r_2 r \le 306$, and we have $r_2 = 0$ if $r > 303$. However, if $r > 303$, then $r_2 = 0$ and we have $x = 1$ by Lemma \ref{2-2-lem-1}. But by Catalan's conjecture we have $x\ge 2$ --- a contradiction!
Hence $r \le 303$. This proves assertion (ii).
\end{proof}

\begin{lem} \label{2-2-lem-3}
Let $r,s,t \ge 4$ be positive integers, $(x,y,z)$ be a triple of positive coprime integers such that $\delta_r x^r + \delta_s y^s =  z^t$, where $\delta_r,\delta_s \in \{\pm 1\}$.

For each prime number $p\ge 11$, write [cf. Lemma \ref{2-2-lem-1}]
\small\begin{gather*}
a_1(p) \defeq 3 \cdot \frac{p^2+5p}{p^2+p-12}\cdot (1-\frac{1}{p}),  \\
a_2(p) \defeq \max\{3 \cdot \Vol(\mathfrak R_p^{(3)})  + a_1(p)\cdot \log(2), 3 \cdot \Vol(\mathfrak R_p^{(4)}) + a_1(p)\cdot \log(2p) \} .
\end{gather*}\normalsize
Then $3 < a_1(p) \le 4$, and we have $a_2(11) \le 71$, $a_2(13) \le 74$, $a_2(17) \le 80$, $a_2(19) \le 84$, $a_2(23) \le 92$.

Let $l$ be a prime number $\ge 11$.
Write uniquely \small\begin{gather}  \label{2-3-eq0}
x = x_1 \cdot 2^{r_2} \cdot l^{r_l} \cdot x_l^{l}, \;
y = y_1 \cdot 2^{s_2} \cdot l^{s_l} \cdot y_l^{l}, \;
z = z_1 \cdot 2^{t_2} \cdot l^{t_l} \cdot z_l^{l}.
\end{gather}\normalsize
where $x_1, x_l, y_1, y_l, z_1, z_l, 2, l$ are pairwise coprime positive integers, 
$r_2, r_l, s_2, s_l, t_2, t_l \ge 0$, 
$x_l^l = x_{(2l)}^{\lagl}$, $y_l^l = y_{(2l)}^{\lagl}$, $z_l^l = z_{(2l)}^{\lagl}$.
Write $r' \defeq r - a_1(l)$,  $s' \defeq s - a_1(l)$,  $t' \defeq r - a_1(l)$.

(i) Suppose that $l \nmid rst$, then we have
\small\begin{align} \label{2-3-eq1}
r'\cdot \log(x_1) + s'\cdot \log(y_1) + t'\cdot \log(z_1) \le a_2(l).
\end{align}\normalsize
Hence \small\begin{align}  \label{2-3-eq3}
(r' + s') \cdot \min\{\log(x_1), \log(y_1) \} + t' \cdot \log(z_1) \le a_2(l).
\end{align}\normalsize

(ii) Suppose that $l \nmid rst$, and assume without loss of generality that $4 \le r \le s \le t$.
Then $t' \ge s' \ge 1$, $r' \ge 0$, and $r' = 0$ if and only if $(r,l) = (4,11)$.
Moreover, when $t' \le r' + s'$, we have
\small\begin{align}  \label{2-3-eq2}
\min\{\log(x_1 y_1), \log(x_1 z_1), \log(y_1 z_1) \} \le  \frac{2a_2(l)}{ r' + s' + t' }.
\end{align}\normalsize

(iii) Suppose that $l \nmid rs$, then we have
\small\begin{align} \label{2-3-eq1a}
r'\cdot \log(x_1) + s'\cdot \log(y_1) \le a_2(l), \quad
\min\{\log(x_1), \log(y_1)\} \le \frac{a_2(l)}{r' + s'}.
\end{align}\normalsize

\end{lem}
\begin{proof}
assertion (i) is consequence of Lemma \ref{2-1-lem-1}. 
When $l\mid xyz$, following the notation in Lemma \ref{2-1-lem-1},
we have $N_l = (x_l^{2r} y_l^{2s} z_l^{2t})^l$, $N'_l = x_1^{2r} y_1^{2s} z_1^{2t}$, $\log\rad(2^{-8} x^{2r} y^{2s} z^{2t} / N_l) \le \log\rad(2l x_1y_1z_1)$.
Then by (\ref{2-2-eq1}) we have
\small\begin{align*}  
\log(x_1^{r} y_1^{s} z_1^{t}) \le&  a_1(l) \cdot \log\rad(2^{-8} x^{2r} y^{2s} z^{2t} / N_l)  + 3\cdot \Vol(\mathfrak R_l^{(4)}) 
\le a_1(l) \cdot \log\rad(x_1y_1z_1) + a_2(l) .
\end{align*}\normalsize
When $l\nmid xyz$, we have $\log\rad(2^{-8} x^{2r} y^{2s} z^{2t} / N_l) \le \log\rad(2 x_1y_1z_1)$. Then we can prove $ \log(x_1^{r} y_1^{s} z_1^{t}) \le a_1(l) \cdot \log\rad(x_1y_1z_1) + a_2(l)$ similarly.

Hence we have \small\begin{align*}  
a_2(l) &\ge \log(x_1^{r} y_1^{s} z_1^{t}) - a_1(l) \cdot \log(x_1y_1z_1) = r'\cdot \log(x_1) + s'\cdot \log(y_1) + t'\cdot \log(z_1) 
\\  &\ge (r' + s') \cdot \min\{\log(x_1), \log(y_1) \} + t' \cdot \log(z_1) .
\end{align*}\normalsize
This proves assertion (i).

For assertion (ii), since the signature $(2,4,n)$ or its permutation is solved for $n\ge 4$, we have $s,t \ge 5$, $3 < a_1(l) \le 4$, and $a_1(l) = 4$ if and only if $l=11$. Hence $t' \ge s' \ge 1$, $r' \ge 0$, and $r' = 0$ if and only if $(r,l) = (4,11)$. 

When $t' \le r' + s'$, to prove (\ref{2-3-eq2}), let $A \defeq \min\{\log(x_1y_1), \log(x_1z_1), \log(y_1z_1) \}$.
Then by (\ref{2-3-eq1}), we have $2 a_2(l) \ge (r'+s'-t')\log(x_1 y_1) + (r'+t'-s')\log(x_1 z_1) + (s'+t'-r')\log(x_1 y_1) \ge (r'+s'-t')A + (r'+t'-s')A + (s'+t'-r')A = (r'+s'+t')A$, hence $A \le \frac{2 a_2(l)}{r'+s'+t'}$.

The proof of assertion (iii) is similar to that of assertion (i).
\end{proof}

\begin{remark} \label{2-2-rmk-1}
(i) Suppose that $l\nmid rs$ and consider the unique decomposition $x = x_1 \cdot 2^{r_2} \cdot l^{r_l} \cdot x_l^{l}$. 
In Lemma \ref{2-2-lem-1} and Lemma \ref{2-2-lem-2}, we have proved that $r\le 313$, $r_2 \le \frac{306}{r}$, $r_l \le \frac{37}{r}$;  
$x_l \in \{1,3,5\}$ if $(r,l) = (4,11)$,
$x_l \in \{1,3\}$ if $(r,l) = (4,13)$, $(4,17)$, $(5,11)$, $(5,13)$, $(6,11)$,
and $x_1 = 1$ otherwise. 
We have similar results for $y$ [and for $z$ if $l\nmid t$], and we have $x,y,z\ge 2$, $\min\{r_2,s_2,t_2\} = 0$, $\min\{r_l,s_l,t_l\} = 0$, $\min\{x_l,y_l\} = 1$. 
Combined with the upper bounds of $x_1, y_1, z_1$ in Lemma \ref{2-2-lem-3}, we can try to search for all possible positive coprime integers $(x,y,z)$ satisfying  $\delta_r x^r + \delta_s y^s =  z^t$, where $\delta_r,\delta_s \in \{\pm 1\}$.

(ii) To find all non-trivial primitive solutions of $x^r + y^s = z^t$, it suffices to find all possible positive coprime integers $(x,y,z)$ satisfying  $\delta_r x^r + \delta_s y^s =  z^t$, where $\delta_r,\delta_s \in \{\pm 1\}$. 
Note that this equation is ``symmetric'' for $r$, $s$ and $t$, hence we can assume without loss of generality that $4\le r\le s\le t$.
We shall apply the unique decomposition  (\ref{2-3-eq0}) to the possible solution $(x,y,z)$, for prime number $l \ge 11$ chosen below.

When $l\nmid rst$ and any two of $x_1, y_1, z_1$ are given, suppose that $x_1,y_1$ are given as an example, then there are only finitelly many possible $(x,y)$ [cf. the discussion in (i)]. To solve the equation $\delta_r x^r + \delta_s y^s =  z^t$ with given $(x_1, y_1)$, we only need to check whether $|x^r \pm y^s|$ is a $t$-th power for each possible $(x,y,r,s,t)$. We shall call this procedure as check for $(x_1,r, y_1,s,l,t)$.

When $t \ge 70$, we have $z = 2^{t_2}$, where $t_2 \ge 1$ [since $z\ge 2$] and $70\le t \le t_2 t \le 306$  [since $t_2\le \frac{306}{t}$]. If $s\ge 70$, then $y = 2^{s_2}$ is also even., which contradicts to $\gcd(x,y,z)=1$, hence we have $4\le r \le s < 70$.
In the case of $t\ge 70$, when $l\nmid rs$ and  any one of $x_1, y_1$ is given, suppose that $y_1$ is given as an example, then there are only finitelly many possible $(y, z^t)$ [since $z^t = 2^m$, where $70 \le m \le 306$]. To solve the equation $\delta_r x^r + \delta_s y^s =  z^t$ with given $x_1$, we only need to check whether $|y^s \pm 2^m|$ is a $r$-th power for each possible $(y,m,r,s)$. We shall call this procedure as check for $(x_1,r,l,s)$.

(iii) We can divide the search procedure into several parts:
\begin{itemize}
\item [(P1)] For given $4 \le r \le s < 70$ and all possible $t\ge 70$, choose $l\in \{11,13,17,19\}$, $l\nmid rs$. 
Then by (\ref{2-3-eq1a}), we have $\min\{\log(x_1), \log(y_1)\} \le \frac{a_2(l)}{r' + s'}.$ 
Then for $\log(x_1) \le \frac{a_2(l)}{r' + s'}$, check for $(x_1, r, l, s)$;
for $\log(y_1) \le \frac{a_2(l)}{r' + s'}$  $(y_1, s, l, r)$.

\item [(P2)] For given $4\le r\le s\le t < 70$ such that $t \ge r + s - 3$, choose $l\in \{11,13,17,19\}$, $l\nmid rst$. Since $3 < a_1(l) \le 4$, we have $t' \ge r' + s'$.
Then by (\ref{2-3-eq3}), we have $(r' + s') \cdot \min\{\log(x_1), \log(y_1) \} + t' \cdot \log(z_1) \le a_2(l)$. 
Then for $(r' + s') \cdot \log(x_1) + t' \cdot \log(z_1) \le a_2(l)$, check for $(x_1, r, z_1, t, l, s)$;
for $(r' + s') \cdot \log(y_1) + t' \cdot \log(z_1) \le a_2(l)$, check for $(y_1, a, z_1, t, l, r)$.

\item [(P3)] For given $4\le r\le s\le t < 70$ such that $t \le r + s - 4$, choose $l\in \{11,13,17,19\}$, $l\nmid rst$. Since $3 < a_1(l) \le 4$, we have $t' \le r' + s'$.
Then by (\ref{2-3-eq2}), we have $\min\{\log(x_1 y_1), \log(x_1 z_1), \log(y_1 z_1) \} \le  \frac{2a_2(l)}{ r' + s' + t' }$.
Then for $\log(x_1 y_1) \le \frac{2a_2(l)}{ r' + s' + t' }$, check for $(x_1, r, y_1, s, l, t)$;
for $\log(x_1 z_1) \le \frac{2a_2(l)}{ r' + s' + t' }$, check for $(x_1, r, z_1, t, l, s)$;
for $\log(y_1 z_1) \le \frac{2a_2(l)}{ r' + s' + t' }$, check for $(y_1, s, z_1, t, l, r)$.
\end{itemize}

\end{remark}

\begin{prop} \label{2-2-prop}
For any integers $r,s,t \ge 4$, such that $(r,s,t)$ is not a permutation of $(4,5,n), (4,7,n), (5,6,n) $ with $7 \le n \le 301$, the generalized Fermat equation
$$ x^r + y^s= z^t $$
has no non-trivial primitive solution.
\end{prop}
\begin{proof}
It suffices to find all positive coprime integers $(x,y,z)$ satisfying  $\delta_r x^r + \delta_s y^s =  z^t$, where $\delta_r,\delta_s \in \{\pm 1\}$. 
Note that this equation is ``symmetric'' for $r$, $s$ and $t$, hence we can assume without loss of generality that $4\le r\le s\le t$.

Then when $(r,s) \neq (4,5), (4,7), (5,6)$, the computation in \cite{Diophantine_after_IUT_I}, which is based on the discussion in Remark \ref{2-2-rmk-1} shows that such solution $(x,y,z)$ does not exist.

When $(r,s) = (4,5), (4,7), (5,6)$, by Lemma \ref{2-2-lem-2}, we have $t \le 303$. Then since the signatures $(4,5,2), (4,7,2), (5,6,2), (4,5,3), (2,7,3), (5,3,3)$ have benn solved [cf. Introduction], $t$ cannot divided by $2$ or $3$, hence $t \le 301$. We can prove $t\ge 7$ similarly.
\end{proof}

\section{The signatures ($\mathbf{2,3,t}$)}
Let $t\ge 7$ be a positive integer. In this section, we shall consider about the generalized Fermat equation with signature $(2,3,t)$ or its permutations. 
Since the case of $t\in \{6,7,8,9,10,15\}$ have been solved [cf. Introduction], we only need to consider for the case of $t\ge 11$ and $6,7,8,9,10,15 \nmid t$. Then $t$ must be a multiple of some prime number $q \ge 5$.

\subsection{Upper bounds}

\begin{lem}  \label{3-1-lem-1}
Let $t\ge 7$ be a positive integer, $(x,y,z)$ be a triple of positive coprime integers such that $\delta_2 x^2 + \delta_3 y^3 = z^t$, where $\delta_2,\delta_3 \in \{\pm 1\}$.
Write $z_1 \defeq z_{(6)} = 2^{-v_2(z)} 3^{-v_3(z)} z$ for the coprime to $2\cdot 3$ part of $z$.
 
(i) If $z_1 < 19$, then $(\delta_2x^2, \delta_3y^3, z^t)$ belongs to one of the following:
\begin{gather*}
(3^2, -2^3, 1), (71^2, -17^3, 2^7), (13^2, 7^3, 2^9), (-1549034^2, 15613^3, 33^8), (21063928^2, -76271^3, 17^7).
\end{gather*}
(ii) Suppose that $t\ge 11$ and $z \ge 2$. 
Then $t$ is not a multiple of $6,7, 8, 9, 10, 15$, hence $t$ is a multiple of some prime number $q \ge 5$; we have $z _1 \ge 13$, and $z_1$ is not a power of $5$, $7$ or $11$.
\end{lem}
\begin{proof}
For assertion (i), write $t_2 \defeq v_2(z)$, $t_3 \defeq v_3(z)$; write $r$ (resp. $r_2$; $r_3$) for the residue of $t$ (resp. $t_2 t$; $t_3 t$) modulo $6$; let $k \ge 0$ (resp. $k_2$; $k_3$) be a integer such that $6k = t - r$ (resp. $6 k_2 = t_2 t - r_2$;  $6 k_3 = t_3 t - r_3$). 
Then we have 
$$(2^{-3 k_2} 3^{-3 k_3} z_1^{-3 k} x)^2 = (-\delta_2\delta_3 2^{-2 k_2} 3^{-2 k_3} z_1^{-2 k} y)^3 + \delta_2 2^{r_2} 3^{r_3} z_1^r .$$
Hence $P = (-\delta_2\delta_3 2^{-2 k_2} 3^{-2 k_3} z_1^{-2 k} y, 2^{-3 k_2} 3^{-3 k_3} z_1^{-3 k} x)$ is a rational point on the elliptic curve $Y^2 = X^3 + \delta_2 2^{r_2} z_1^r$, whose denominators of coordinates are divided by $z_1$, and are divided by $2$ (resp. $3$) when $r_2 > 0$ (resp. $r_3 > 0$).
When $z_1 < 19$, since $\gcd(6, z_1) = 1$, we have $z_1 \in \{1,5,7,11,13,17\}$.

When $r > 0$, for $\delta_2 \in \{\pm 1\}$, $z_1 \in \{1,5,7\}$, $0\le r_2, r_3 \le 5$, $1\le r\le 5$, let $S$ be the set of all prime divisors of $6z_1$. 
Using Magma \cite{Magma}, we can find all the possible $P$ by computing the $S$-integral points on the elliptic curve $Y^2 = X^3 + \delta_2 2^{r_2} 3^{r_3} z_1^r$,
and then choosing the $S$-integral points whose denominators of coordinates are divided by $z_1$, and are divided by $2$ (resp. $3$) when $r_2 > 0$ (resp. $r_3 > 0$).
These possible $S$-integral points imply that $(\delta_2x^2, \delta_3y^3, z^t) = (3^2, -2^3, 1)$, $(71^2, -17^3, 2^7)$, $(13^2, 7^3, 2^9)$, $(-1549034^2, 15613^3, 33^8)$, $(21063928^2, -76271^3, 17^7)$. 
For the Magma codes, cf. \cite{code_of_zpzhou}.

When $r = 0$, we have $6\mid t$. Then since the only positive coprime solution to $\pm x^2 \pm y^3 = (z^{t/6})^6 = (z')^6$ is the Catalan solution $3^2-2^3=1^6$, we can deduce that $z^{t/6} = 1$, hence $(\delta_2x^2, \delta_3y^3, z^t) = (3^2, -2^3, 1)$.
This proves assertion (i).

Since all the solutions in (i) have $z=1$ or $t < 11$, and the generalized Fermat equation has been solved for the signature $(2,3,n)$, where $n\in \{6,7,8,9,10,15\}$ [cf. Introduction]. Assertion (ii) follows directly from these observations.
\end{proof}

\begin{lem}  \label{3-1-lem-2}
Let $t\ge 11$ be a positive integer, $(x,y,z)$ be a triple of positive coprime integers such that $z\ge 2$ and $\delta_2 x^2 + \delta_3 y^3 = z^t$, where $\delta_2,\delta_3 \in \{\pm 1\}$. Let $l\ge 11$ be a prime number such that $l\neq 13$.

(i) Let $(a,b,c) = (x, \delta_2\delta_3 y, \delta_2 z^t)$, then $a^2+b^3 = c$, $\gcd(a,b,c) = 1$.
Let $E$ be the elliptic curve defined over $\Q$ by the equation 
$$Y^2 = X^3 + 3bX + 2a,$$
then the $j$-invariant
$$j(E) = \frac{ 1728\delta_3 y^3 }{ z^t } \notin J \defeq \{0,2^6\cdot 3^3,2^2\cdot 73^3\cdot 3^{-4},2^{14}\cdot 31^3\cdot 5^{-3}\} .$$
Write 
$$N \defeq z^t / \gcd(1728, z^t), \; N_l \defeq N^{\lagl}, \; N'_l \defeq (N/N_l)_{(l)}, \;  N'_{2l} \defeq (N/N_l)_{(2l)}. $$
Then $x,y,z \ge 2$, $N$ is the denominator of $j(E)$ and $N$ is not a power of $2$.

(ii) Write $F\defeq \Q(E[12])$,
then when $N'_l \neq 1$, by the construction in \cite{Diophantine_after_IUT_I}, Proposition 2.5, there exists a $\mu_6$-initial $\Theta$-data $\mathfrak D = \mathfrak D(E,F,l,\mu_6)$ of type $(l, N, N'_l)$.

Let $\mathfrak{R}_l^{(1)}$ be the ramification dataset [cf. \cite{Diophantine_after_IUT_I}, Definition 1.12] consists of the following data:
\begin{gather*}
l_0 = l, \;  e_0 = 12, \; S_0=\{2,3,l\}, \; S_{\gen}^{\multi} = \{e: e\mid 12 l \},  \;
S_2^{\good} = \{e: e\mid 2^8\cdot 3^2, 2\mid e\}, \\ 
S_3^{\good} = \{e: e\mid 2^6\cdot 3^2, 2\mid e\}, \;
S_l^{\good} = \{l-1, l(l-1), l^2-1\}, \\
S_2^{\multi} = \{e: e\mid 24l, 2\mid e\}, \; S_3^{\multi} = \{e\mid 48l, 2\mid e\}, \;
S_l^{\multi} = \{(l-1)\cdot e: e\mid 12l\} .
\end{gather*}
Then $\mathfrak D$ admits $\mathfrak R_l^{(1)}$.
Hence we have 
\begin{equation} \small \label{3-1-eq1}
    \frac{1}{6}\log(N'_l) \le  \frac{l^2+5l}{l^2+l-12}\cdot \big( (1-\frac{1}{12  l})\cdot \log\rad(N) - \frac{1}{12}(1- \frac{1}{l}) \log\rad(N_l)  \big) + \Vol(\mathfrak R_l^{(1)}).
\end{equation}
The above inequality holds for $N'_l = 1$.

(iii) By Lemma \ref{3-1-lem-1}, we can choose a prime number $q\ge 5$, such that $q\mid t$. 
Write $F\defeq \Q(\sqrt{-1}, E[2q])$, and suppose that $l\nmid q(q^2-1)$, 
then when $N'_{2l} \neq 1$, by the construction in Proposition \ref{1-prop: construction of mu_2 initial Theta-data}, there exists a $2$-torsion initial $\Theta$-data $\mathfrak D = \mathfrak D(E,F,l,2\text{-tor})$ of type $(l, N, N'_{2l})$.

Let $\mathfrak{R}_{l,q}^{(2)}$ be the ramification dataset [cf. \cite{Diophantine_after_IUT_I}, Definition 1.12] consists of the following data:
\begin{gather*}
l_0 = l, \;  e_0 = 2, \; S_0=\{2,3,q,l\}, \; S_{\gen}^{\multi} = \{1,2,l,2l\}, \\
S_2^{\good} = \{e: e\mid 2^5\cdot 3^2, 2\mid e\}, \;
S_3^{\good} = \{1\}, \;
S_q^{\good} = \{q-1, q(q-1), q^2-1\}, \\
S_l^{\good} = \{l-1, l(l-1), l^2-1\}, \;
S_2^{\multi} = \{e: e\mid 8ql, 2\mid e\}, \; S_3^{\multi} = \{e\mid 2ql, 2\mid e\}, \\
S_q^{\multi} = \{(q-1)\cdot e: e\mid 2l\}, \; S_l^{\multi} = \{(l-1)\cdot e: e\mid 2l\} .
\end{gather*}
Then $\mathfrak D$ admits $\mathfrak R_{l,q}^{(2)}$.
Hence we have 
\begin{equation} \small \label{3-1-eq2}
    \frac{1}{6}\log(N'_{2l}) \le  \frac{l^2+5l}{l^2+l-12}\cdot \big( (1-\frac{1}{2  l})\cdot \log\rad(N) - \frac{1}{2}(1- \frac{1}{l}) \log\rad(N_l)  \big) + \Vol(\mathfrak R_l^{(2)}).
\end{equation}
The above inequality holds for $N'_{2l} = 1$.

\end{lem}
\begin{proof}
For assertion (i), $j(E) = \frac{ 1728\delta_3 y^3 }{ z^t }$ follows from Corollary \ref{1-2-cor2}, (i). Then since $[z]_2 \ge 13$ by Lemma \ref{3-1-lem-1}, we can see that $j(E) \not\in J$, and $N$ is not a power of $2$.

For assertion (ii), since $j(E) \notin J$, $N$ is not a power of $2$, $N'_l \neq 1$, $E_F$ is semi-stable [cf. Corollary \ref{1-2-cor0}], $l\ge 11$ and $l\neq 13$, $F = \Q(E[12])$ is Galois over $\Q$ [cf. \cite{Diophantine_after_IUT_I}, Proposition 2.1],
by the construction in \cite{Diophantine_after_IUT_I}, Proposition 2.5, there exists a $\mu_6$-initial $\Theta$-data $\mathfrak D = \mathfrak D(E,F,l,\mu_6)$ of type $(l, N, N'_l)$. Then by Corollary \ref{1-2-cor0}, we can see that $\mathfrak D$ admits $\mathfrak R_l^{(1)}$, hence by Corollary \ref{1-cor: The log-volume of ramification dataset} we have (\ref{3-1-eq1}). Since the log-volume of a ramification dataset is non-negative, we have $\Vol(\mathfrak R_l^{(1)}) \ge 0$, hence  (\ref{3-1-eq1}) holds for $N'_l = 1$.

For assertion (iii), by using Corollary \ref{1-2-cor2} instead of Corollary \ref{1-2-cor0}, and using Proposition \ref{1-prop: construction of mu_2 initial Theta-data} instead of \cite{Diophantine_after_IUT_I}, Proposition 2.5,
assertion (iii) can be proved similarly to the proof of assertion (ii).

\end{proof}

\begin{lem}  \label{3-1-lem-3}
Let $t\ge 11$ be a positive integer, $(x,y,z)$ be a triple of positive coprime integers such that $z\ge 2$ and $\delta_2 x^2 + \delta_3 y^3 = z^t$, where $\delta_2,\delta_3 \in \{\pm 1\}$.

Let $S$ be a finite set of prime numbers with cardinality $n = |S| \ge 2$, such that for each $l\in S$, we have $l \ge 11$, $l\neq 13$;
let $2\le k\le S$ be a positive integer, $p_0$ be the smallest prime number in $S$, $k(S)$ be the product of $k$ smallest prime numbers in $S$;
let $n_0 = \gcd(1728, z^t)$, $N = n_0^{-1} z^t$;

for each $l\in S$, write $\Vol(l)\defeq \Vol(\mathfrak R_l^{(1)})$ [cf. Lemma \ref{3-1-lem-2}]; let $e_0 = 12$, $S_1 = \{2,3\}$,  $p_B = 2$, $n_k(S)$ be the product of the smallest $k$ prime numbers in $S$.

Let $S$ be a finite set of prime numbers with cardinality $n = |S| \ge 2$, such that for each $l\in S$, we have $l \ge 11$;
let $2\le k\le S$ be a positive integer, $p_0$ be the smallest prime number in $S$, $k(S)$ be the product of $k$ smallest prime numbers in $S$;
let $n_0 = 2^8$, $N = n_0^{-1}x^{2r}y^{2s}z^{2t}$, $u_0 = 8$, $p_N = 2$, $e_0 = 3$,  $S_1 = \{2\}$, $n_1(S) = 2 p_0$, $n_k(S) = 2 k(S)$;
for each $l\in S$, write $\Vol(l) \defeq \Vol(\mathfrak R_l^{(1)})$, cf. Lemma \ref{2-1-lem-1}, (ii).
Suppose that one of the following situations is satisfied:

\begin{itemize}
\item[(i)] Let $u_0 \ge 11$ be a positive integer, such that $u_0\le t$;
let $n_1(S) = u_0$.

\item[(ii)] Suppose that for each $l\in S$, we have $l\nmid t$.
Let $u_0 = t$, $n_1(S) = p_0 t$.
\end{itemize}

Then with the notation in (i) or (ii), the assumptions in Definition \ref{1-3-def: basic notation for global inequalities} are satisfied, and we can define $b_1$, $b_2$.
Note that with the notation in (i), the definitions of $b_1$, $b_2$ are independent to $t\ge u_0$.
Hence if we have $b_1 <  1$ and $n_k(S)\cdot \log(2) > \frac{b_2}{1-b_1}$, then by Lemma \ref{1-3-lem: local-global inequalities}, $\log(N)$ cannot belong to the interval $(\frac{b_2}{1-b_1}, n_k(S)\cdot \log(2) ) .$
\end{lem}
\begin{proof}
The proof follows from Lemma \ref{3-1-lem-2} and is elementary, which is similar to the proof of Lemma \ref{2-1-lem-2}.
\end{proof}

\begin{lem}  \label{3-1-lem-4}
Let $t\ge 11$ be a positive integer, $(x,y,z)$ be a triple of positive coprime integers such that $z\ge 2$ and $\delta_2 x^2 + \delta_3 y^3 = z^t$, where $\delta_2,\delta_3 \in \{\pm 1\}$.

Let $q$ be a prime number $\ge 5$, such that $q\mid t$. 
Let $S$ be a finite set of prime numbers with cardinality $n = |S| \ge 2$, such that for each $l\in S$, we have $l \ge 11$, $l\neq 13$ and $l\nmid q(q^2-1)$.
let $p_0$ be the smallest prime number in $S$; $2\le k\le S$ be a positive integer; $n_0 = \gcd(27, z^t)$, $N \defeq n_0^{-1} [z^t]_2$, $e_0 = 2$;
for each $l\in S$, write $\Vol(l)\defeq \Vol(\mathfrak R_l^{(2)})+\log(2)$ [cf. Lemma \ref{3-1-lem-2}].

Let $u_0 = t$, $S_1 = \{3\}$, $n_1(S) = t$, $p_B = 3$, $n_k(S)$ be the product of the smallest $k$ prime numbers in $S$.
Then the assumptions in Definition \ref{1-3-def: basic notation for global inequalities} are satisfied, and we can define $b_1$, $b_2$. 
Hence if we have $b_1 <  1$ and $n_k(S)\cdot \log(3) > \frac{b_2}{1-b_1}$, then by Lemma \ref{1-3-lem: local-global inequalities}, $\log(N)$ cannot belong to the interval $(\frac{b_2}{1-b_1}, n_k(S)\cdot \log(3) ) .$
\end{lem}
\begin{proof}
Since $\log\rad(N) \le \log\rad([N]_2) + \log(2)$, the proof follows from Lemma \ref{3-1-lem-3} and is elementary, which is similar to the proof of Lemma \ref{2-1-lem-2}.
\end{proof}

\begin{prop} \label{3-1-prop}
Let $t\ge 11$ be a positive integer, $(x,y,z)$ be a triple of positive coprime integers such that $z\ge 2$ and $\delta_2 x^2 + \delta_3 y^3 = z^t$, where $\delta_2,\delta_3 \in \{\pm 1\}$. Write $z_1 \defeq 2^{-v_2(z)} z$ for the coprime to $2$ part of $z$.

(i) We have $z_1 \ge 13$; the exponent $t$ is not divided by $6,7,8,9,10,15$, hence We have $t \in \{11,13\}$ or $17\le t < 1000$.

(ii) We have $\log(z^t) < 3730$ for $t = 11$; $\log(z^t) < 3085$ for $t = 13$; and $\log(z^t) < 2789$ for $t \ge 17$.

(iii) We have $t\le 353$ or $t=373$. In addition, we have 
$\log(z_1^t) < 1183$ for $t = 11$; $\log(z_1^t) < 994$ for $t = 13$; 
$\log(z_1^t) < 919$ for $t = 17$; $\log(z_1^t) < 895$ for $t = 19$; 
$\log(z_1^t) < 1019$ for $t = 373$; and  $\log(z_1^t) < 981$ for $t \ge 17$, $t\neq 373$.
\end{prop}
\begin{proof}
For assertion (i), by Lemma \ref{3-1-lem-1}, we have $z_1 \ge 13$; the exponent $t$ is not divided by $6,7,8,9,10,15$, hence $t \in \{11,13\}$ or $t \ge 17$.
Then by Lemma \ref{3-1-lem-3}, (i) and the computation in \cite{code_of_zpzhou}, we have $\log(z^t) < 2560$ for $t \ge 1000$. But since $z\ge z_1 \ge 13$, if $t\ge 1000$, then $\log(z^t) \ge 1000 \cdot \log(13) > 2564$ --- a contradiction! Hence $t < 1000$.

assertion (ii) follows from Lemma \ref{3-1-lem-3}, (ii) and the computation in \cite{code_of_zpzhou}.
assertion (iii) follows from Lemma \ref{3-1-lem-4} and the computation in \cite{code_of_zpzhou}.
\end{proof}

\begin{cor} \label{3-1-cor}
For any integer $n\ge 3$, there does not exist any triple $(x, y, z)$ of non-zero coprime integers that satisfies the generalized Fermat equation
$$ x^2 + y^{2n}= z^3 .$$
\end{cor}
\begin{proof}
Suppose that for some $n \ge 3$, there exists a triple of non-zero coprime integers $(x, y, z)$ such that $ x^2 + y^{2n} = z^3 .$ Then we have $n > 10^7$, cf. Chen \cite{Chen-21}, Dahmen \cite{Dahmen-29}. 
Hence by Proposition \ref{3-1-prop}, we have $|y|=1$, $ x^2 + 1 = y^3$.
However, the only integral points on the elliptic curve $Y^2 = X^3 - 1$ are $(1,0)$ and the point at infinity. Hence $(x,y) = (0,1)$, which is impossible!
Hence  the generalized Fermat equation
$ x^2 + y^{2n}= z^3 $ has no non-trivial primitive solution.
\end{proof}

\subsection{Structure of solutions}

We shall analyze the structure of the solution $(x,y,z)$ in more detail.

\begin{lem}  \label{3-2-lem-1}
Let $t\ge 60$ be a positive integer, $(x,y,z)$ be a triple of positive coprime integers such that $z\ge 2$ and $\delta_2 x^2 + \delta_3 y^3 = z^t$, where $\delta_2,\delta_3 \in \{\pm 1\}$. 
Let $q \ge 5$ be a prime number, such that $q\mid t$. 

Let $S$ be the set of prime numbers $p \ge 11$, such that $p\neq 13$, $p\nmid t$ and $p\nmid (q^2-1)$. For each $p\in S$, denote [cf. Lemma \ref{4-1-lem-1}]
\small\begin{align*}
a_1(p) \defeq 6 \cdot \frac{p^2+5p}{p^2+p-12}\cdot (1-\frac{1}{2p}),  \;
a_2(p) \defeq 6 \cdot \Vol(\mathfrak R_{l,q}^{(2)})  + a_1(p)\cdot \log(6p) .
\end{align*}\normalsize

Let $l\in S$, write uniquely \small\begin{gather}  \label{3-2-eq1}
z = z_1 \cdot 2^{t_2} \cdot 3^{t_3} \cdot l^{t_l} \cdot z_l^{l},
\end{gather}\normalsize
where $z_1, z_l, 2, 3, l$ are pairwise coprime positive integers, $t_2, t_3, t_l \ge 0$, $z_l^l = z_{(6l)}^{\lagl}$.

(i) We have $z_l = 1$, hence
\small\begin{align}  \label{3-2-eq2}
a_2(l) \ge t \cdot \log(z_1) - a_1(l) \cdot \log\rad(z_1) \ge  (t - a_1(l)) \cdot \log(z_1).
\end{align}\normalsize

(ii) Let $p$ be a prime divisor of $z$. 
When $t\ge 93$, we have $p \le 31$; when $t\ge 110$ and $t\neq 113, 121$, we have $p\le 17$.

(iii) When $t\ge 95$, we have $z_{(6)} < 43$; when $t \ge 107$, we have $z_{(6)} < 35$.

(iv) We have $t\le 109$ or $t = 113, 121$.
\end{lem}
\begin{proof}
For assertion (i), by Proposition \ref{3-1-prop}, we have $1019 > \log(z_{(2)}^t) \ge z_l^{tl}$, hence $\log(z_l) < \frac{1019}{t l } \le \frac{1019}{60 \cdot 11} < \log(5)$, $z_l < 5$. Then since $\gcd(6l, z_l) = 1$, we have $z_l = 1$. 
Hence $\log(z_1^t) \le \log(N'_{2l})$, $\log\rad(N) \le \log\rad(6l z_1)$, and  (\ref{3-2-eq2}) follows from (\ref{3-1-eq2}).

For assertion (ii), choose $l\in S$ such that $l\neq p$, then we have $p\mid z_1$. 
Hence by (\ref{3-2-eq2}), we have $\log(p) \le \log(z_1) \le \frac{a_2(l)}{t - a_1(l)}$.
Recall that $t < 400$ [cf. Proposition \ref{3-1-prop}, (iii)]. 
For each $95 \le t < 400$, we can choose such $l$ explicitly, and get upper bounds for each possible $p$. Then assertion (ii) follows from the computation in \cite{code_of_zpzhou}.

For assertion (iii), we can choose $l\in S$ such that $l\nmid z_1$ explicitly, then by (\ref{3-2-eq2}), we have $\log(z_1) \le \frac{a_2(l)}{t - a_1(l)}$. Then assertion (ii) follows from the computation in \cite{code_of_zpzhou}.

For assertion (iv), assume that $t \ge 110$ and $t\neq 113, 121$, then by (ii),  $z_{(6)}$ only has prime divisors $\in \{5,7,11,13,17\}$. Then since by (iii) we have $z_{(6)} < 35$, hence $z_{(6)} \in \{1, 5, 7, 11, 13, 17, 25\}$. But this contradicts Lemma \ref{3-1-lem-1}, (ii), hence assertion (iv) holds.
\end{proof}

\begin{cor} \label{3-2-cor}
Suppose that $(r,s,t)$ is a permutation of $(2,3,n)$ with $n\ge 7$, such that $n$ does not satisfy the following conditions:
\begin{itemize}
\item $11\le n \le 109$ or $n\in \{113, 121\}$.
\item $n$ is not divided by $6,7,8,9,10,15$.
\end{itemize}
Then the generalized Fermat equation $x^r + y^s = z^t$ has no non-trivial primitive solution except for the solutions related to the  Catalan solutions and the nine non-Catalan solutions enumerated in the introduction.
\end{cor}
\begin{proof}
Suppose that $(x,y,z)$ is a non-trivial primitive solution to the equation $x^r + y^s = z^t$, such that $|x|, |y|, |z| \ge 2$.
By interchanging $(r,s,t)$, we can assume without loss of generality that $(r,s,t) = (2,3,n)$, and $(x,y,z)$ is a triple of positive coprime integers such that $z\ge 2$ and $\delta_2 x^2 + \delta_3 y^3 = z^t$, where $\delta_2,\delta_3 \in \{\pm 1\}$. 

Since the signatures $(2,3,n)$ with $n\in\{6,7,8,9,10,15\}$ have been solved [cf. Introduction], we have $n\ge 11$ and $n$ is not divided by $6,7,8,9,10,15$.
Then by Lemma \ref{3-2-lem-1}, we have $n \le 109$ or $n\in \{113, 121\}$.
\end{proof}

\section{The signatures ($\mathbf{3,r,s}$)}

Let $r,s\ge 3$ be positive integers. In this section, we shall consider about the generalized Fermat equation with signature $(3,r,s)$ or its permutation.
Since the signatures $(3,3,n)$ and its permutations have been solved for $n\ge 3$ [cf. \cite{Diophantine_after_IUT_I}],  and the signatures $(3,4,5), (3,5,5)$ and their permutations have been solved [cf. Introduction],
we only need to consider for $r,s \ge 4$, $\max\{r,s\} \ge 7$ and $3\nmid rs$.

\subsection{Upper bounds}
In this section, we shall assume that $r,s \ge 4$ are positive integers, $(x,y,z)$ is a triple of positive coprime integers such that $\delta_r x^r + \delta_s y^s = z^3$, where $\delta_r,\delta_s \in \{\pm 1\}$.

\begin{lem}  \label{4-1-lem-1}
Let $r,s \ge 4$ be positive integers, $(x,y,z)$ be a triple of positive coprime integers such that $\delta_r x^r + \delta_s y^s = z^3$, where $\delta_r,\delta_s \in \{\pm 1\}$. 
Let $l\ge 11$ be a prime number such that $l\neq 13$.

(i) Let $(a,b,c) = (\delta_r x^r, \delta_s y^s, z)$, then $a + b = c^3$, $\gcd(a,b,c) = 1$.
Let $E$ be the elliptic curve defined over $\Q$ by the equation 
$$Y^2 + 3cXY + aY= X^3,$$
then $E$ is an elliptic curve with 
$$ j(E) = \frac{27 c^3 (a+9b)^3}{a^3 b} = \frac{27 c^3 (a+9b)^3}{\delta_r \delta_s x^{3r} y^s} \notin J \defeq \{0,2^6\cdot 3^3,2^2\cdot 73^3\cdot 3^{-4},2^{14}\cdot 31^3\cdot 5^{-3}\} .$$
Write 
$$N \defeq x^{3r} y^s / \gcd(27, x^{3r} y^s), \; N_l \defeq N^{\lagl}, \; N'_l \defeq (N/N_l)_{(l)}, \;  N'_{2l} \defeq (N/N_l)_{(2l)}. $$
Then $x,y,z \ge 2$, $N$ is the denominator of $j(E)$ and $N$ is not a power of $2$.

(ii) Write $F\defeq \Q(E[12])$,
then when $N'_l \neq 1$, by the construction in \cite{Diophantine_after_IUT_I}, Proposition 2.5, there exists a $\mu_6$-initial $\Theta$-data $\mathfrak D = \mathfrak D(E,F,l,\mu_6)$ of type $(l, N, N'_l)$.

Let $\mathfrak{R}_l^{(1)}$ be the ramification dataset [cf. \cite{Diophantine_after_IUT_I}, Definition 1.12] consists of the following data:
\begin{gather*}
l_0 = l, \;  e_0 = 12, \; S_0=\{2,3,l\}, \; S_{\gen}^{\multi} = \{e: e\mid 12 l \},  \;
S_2^{\good} = \{e: e\mid 2^8\cdot 3^2, 2\mid e\}, \\ 
S_3^{\good} = \{e: e\mid 2^6\cdot 3^2, 2\mid e\}, \;
S_l^{\good} = \{l-1, l(l-1), l^2-1\}, \\
S_2^{\multi} = \{e: e\mid 24l, 2\mid e\}, \; S_3^{\multi} = \{e\mid 48l, 2\mid e\}, \;
S_l^{\multi} = \{(l-1)\cdot e: e\mid 12l\} .
\end{gather*}
Then $\mathfrak D$ admits $\mathfrak R_l^{(1)}$.
Hence we have 
\begin{equation} \small \label{4-1-eq1}
\frac{1}{6}\log(N'_l) \le  \frac{l^2+5l}{l^2+l-12}\cdot \big( (1-\frac{1}{12  l})\cdot \log\rad(N) - \frac{1}{12}(1- \frac{1}{l}) \log\rad(N_l) \big) + \Vol( \mathfrak R_l^{(1)} ).
\end{equation}
The above inequality holds for $N'_l = 1$.

(iii) Write $F\defeq \Q(E[4])$,
then when $N'_{2l} \neq 1$, by the construction in Proposition \ref{1-prop: construction of mu_2 initial Theta-data}, there exists a $2$-torsion initial $\Theta$-data $\mathfrak D = \mathfrak D(E,F,l,2\text{-tor})$ of type $(l, N, N'_{2l})$.

Let $\mathfrak{R}_l^{(2)}$ be the ramification dataset [cf. \cite{Diophantine_after_IUT_I}, Definition 1.12] consists of the following data:
\begin{gather*}
l_0 = l, \;  e_0 = 4, \; S_0=\{2,3,l\}, \; S_{\gen}^{\multi} = \{e: e\mid 4 l \},  \;
S_2^{\good} = \{e: e\mid 96, 2\mid e\}, \\ 
S_3^{\good} = \{e: e\mid 12, 2\mid e\}, \;
S_l^{\good} = \{l-1, l(l-1), l^2-1\}, \\
S_2^{\multi} = \{e: e\mid 8 l, 2\mid e\}, \; S_3^{\multi} = \{e\mid 8 l, 2\mid e\}, \;
S_l^{\multi} = \{(l-1)\cdot e: e\mid 4 l\} .
\end{gather*}
Let $\mathfrak{R}_l^{(3)}$ be the ramification dataset obtained by setting $S_l^{\good} = \emptyset$ in the definition of $\mathfrak{R}_l^{(2)}$.
Then $\mathfrak D$ admits $\mathfrak R_l^{(3)}$; and if $l\mid xy$, then $\mathfrak D$ admits $\mathfrak R_l^{(3)}$.
Hence we have 
\begin{equation} \small \label{4-1-eq2}
\frac{1}{6}\log(N'_{2l}) \le  \frac{l^2+5l}{l^2+l-12}\cdot \big( (1-\frac{1}{4  l})\cdot \log\rad(N) - \frac{1}{4}(1- \frac{1}{l}) \log\rad(N_l) \big) + \Vol(l).
\end{equation}
where we write $\Vol(l) \defeq \Vol(\mathfrak R_l^{(2)})$ if $l\nmid xyz$, and write $\Vol(l) \defeq \Vol(\mathfrak R_l^{(3)}) \le \Vol(\mathfrak R_l^{(2)})$ if $l\mid xyz$.
The above inequality holds for $N'_{2l} = 1$.

\end{lem}
\begin{proof}
For assertion (i), $ j(E) = \frac{27 c^3 (a+9b)^3}{a^3 b}$ follows from Corollary \ref{1-2-cor3}, (i). Then since $r,s \ge 5$, by Catalan's conjecture we have $x,y,z \ge 2$. Hence we can see that $j(E) \not\in J$, and $N$ is not a power of $2$.

The proof of assertions (ii) and (iii) are similar to that of Lemma \ref{2-1-lem-1}.
\end{proof}

\begin{lem}  \label{4-1-lem-2}
Let $r \ge 4$, $s\ge 7$ be positive integers, $(x,y,z)$ be a triple of positive coprime integers such that $\delta_r x^r + \delta_s y^s = z^3$, where $\delta_r,\delta_s \in \{\pm 1\}$. 

Let $S$ be a finite set of prime numbers with cardinality $n = |S| \ge 2$, such that for each $l\in S$, we have $l \ge 11$, $l\neq 13$;
let $2\le k\le S$ be a positive integer, $p_0$ be the smallest prime number in $S$, $k(S)$ be the product of $k$ smallest prime numbers in $S$;
let $n_0 = \gcd(27, x^{3r} y^s)$, $N \defeq n_0^{-1} x^{3r} y^s$, $7 \le u_0 \le \min\{3r, s\}$, $p_N = 2$, $e_0 = 12$, $S_1 = \{3\}$, $n_1(S) = p_0$, $n_k(S) = k(s)$;
for each $l\in S$, write $\Vol(l)\defeq \Vol(\mathfrak R_l^{(1)})$, cf. Lemma \ref{4-1-lem-1}.
Suppose that one of the following situations is satisfied:

\begin{itemize}
\item [(a)] Let $u'_0 \ge u_0$ be a positive integer, $P\subseteq \{p: v_p(N) \ge u'_0 \} \cup \{p: p\nmid N\}$, $n'_1(S) = \max\{u'_0, 2p_0\}$.
In particular, we can let $u'_0 = \min\{3r, s\}$, $P = \{p: p\mid N\}$, $n'_0 = n_0$ and $N' = N$.

\item [(b)] Let $u'_0 \ge u_0$ be a positive integer, $P\subseteq \{p: v_p(N) \ge u'_0 \} \cup \{p: p\nmid N\}$, $n'_1(S) = \min\{v_p(N): p\in B'\setminus S_1\}$ [for the definition of $B'$, cf. Definition \ref{1-2-def}, (e)], $n'_0 = 27$ if $3\in P$ and $n'_0 = 1$ if $2\notin P$.

\item [(c)] In the situation of (b).
if $l\nmid rs$ for each $l\in S$, then we can let $u'_0 = \min\{3r, s\}$, $n'_1(S) = u'_0 p_0$, $P=\{p: p\mid xy\}$, $N' = (n'_0)^{-1} x^{3r} y^{s}$;
if $l\nmid r$ for each $l\in S$, then we can let $u'_0 = 2r$, $n'_1(S) = u'_0 p_0$, $P=\{p: p\mid x\}$, $N' = (n'_0)^{-1} x^{3r}$;
if $l\nmid s$ for each $l\in S$, then we can let $u'_0 = s$, $n'_1(S) = u'_0 p_0$, $P=\{p: p\mid y\}$, $N' = (n'_0)^{-1} y^{s}$.
\end{itemize}

Then the assumptions in Definition \ref{1-3-def: basic notation for global inequalities} are satisfied, and we can define $b_1$, $b_2$, $b'_1$. 
Hence if $b_1 \le  1$, $b'_1 < 1$ and $n_k(S)\cdot \log(2) > \frac{b_2}{1-b'_1}$, then by Lemma \ref{1-3-lem: local-global inequalities}, $\log(N')$ cannot belong to the interval $(\frac{b_2}{1-b'_1}, n_k(S)\cdot \log(2) ) .$
\end{lem}
\begin{proof}
The proof follows from Lemma \ref{4-1-lem-1} and is elementary, which is similar to the proof of Lemma \ref{2-1-lem-2}.
\end{proof}

\begin{lem}  \label{4-1-lem-3}
Let $r \ge 4$, $s\ge 7$ be positive integers, $(x,y,z)$ be a triple of positive coprime integers such that $\delta_r x^r + \delta_s y^s = z^3$, where $\delta_r,\delta_s \in \{\pm 1\}$. 

Let $S$ be a finite set of prime numbers with cardinality $n = |S| \ge 2$, such that for each $l\in S$, we have $l \ge 11$, $l\neq 13$;
let $2\le k\le S$ be a positive integer, $p_0$ be the smallest prime number in $S$, $k(S)$ be the product of $k$ smallest prime numbers in $S$;
let $n_0 = \gcd(27, x^{3r} y^s)$, $N \defeq (n_0^{-1} x^{3r} y^s)_{(2)}$, $7 \le u_0 \le \min\{3r, s\}$, $p_N = 3$, $e_0 = 4$, $S_1 = \{3\}$, $n_1(S) = p_0$, $n_k(S) = k(s)$;
for each $l\in S$, write $\Vol(l)\defeq \Vol(\mathfrak R_l^{(2)}) + \log(2)$, cf. Lemma \ref{4-1-lem-1}.
Suppose that one of the following situations is satisfied:

\begin{itemize}
\item [(a)] Let $u'_0 \ge u_0$ be a positive integer, $P\subseteq \{p: v_p(N) \ge u'_0 \} \cup \{p: p\nmid N\}$, $n'_1(S) = \max\{u'_0, 2p_0\}$.
In particular, we can let $u'_0 = \min\{3r, s\}$, $P = \{p: p\mid N\}$, $n'_0 = n_0$ and $N' = N$.

\item [(b)] Let $u'_0 \ge u_0$ be a positive integer, $P\subseteq \{p: v_p(N) \ge u'_0 \} \cup \{p: p\nmid N\}$, $n'_1(S) = \min\{v_p(N): p\in B'\setminus S_1\}$ [for the definition of $B'$, cf. Definition \ref{1-2-def}, (e)], $n'_0 = 27$ if $3\in P$ and $n'_0 = 1$ if $2\notin P$.

\item [(c)] In the situation of (b).
if $l\nmid rs$ for each $l\in S$, then we can let $u'_0 = \min\{3r, s\}$, $n'_1(S) = u'_0 p_0$, $P=\{p: p\mid xy\}$, $N' = (n'_0)^{-1} x^{3r} y^{s}$;
if $l\nmid r$ for each $l\in S$, then we can let $u'_0 = 2r$, $n'_1(S) = u'_0 p_0$, $P=\{p: p\mid x\}$, $N' = (n'_0)^{-1} x^{3r}$;
if $l\nmid s$ for each $l\in S$, then we can let $u'_0 = s$, $n'_1(S) = u'_0 p_0$, $P=\{p: p\mid y\}$, $N' = (n'_0)^{-1} y^{s}$.
\end{itemize}

Then the assumptions in Definition \ref{1-3-def: basic notation for global inequalities} are satisfied, and we can define $b_1$, $b_2$, $b'_1$. 
Hence if $b_1 \le  1$, $b'_1 < 1$ and $n_k(S)\cdot \log(2) > \frac{b_2}{1-b'_1}$, then by Lemma \ref{1-3-lem: local-global inequalities}, $\log(N')$ cannot belong to the interval $(\frac{b_2}{1-b'_1}, n_k(S)\cdot \log(3) ) .$
\end{lem}
\begin{proof}
The proof follows from Lemma \ref{4-1-lem-1} and is elementary, which is similar to the proof of Lemma \ref{2-1-lem-2}.
\end{proof}

\begin{prop} \label{4-1-prop}
Let $r,s \ge 4$ be positive integers, $(x,y,z)$ be a triple of positive coprime integers such that $\delta_r x^r + \delta_s y^s = z^3$, where $\delta_r,\delta_s \in \{\pm 1\}$. 

(i) We have $x,y,z \ge 2$, $\max\{r,s\} \ge 7$ and $3 \nmid rs$;
when $r \in \{4,5\}$, we have $7 \le s \le 3677$;
when $r\ge 4$, $s\ge 7$, we have $\log(x^{3r} y^s) < 31000$.

(ii) Suppose that $r,s \ge 7$, Then we have $7\le r,s \le 1226$, $\log(x^r), \log(y^s) \le 1400$,

(iii) Suppose that $r\ge 7$, $s\ge 21$. Let $x_{(2)}, y_{(2)}$ be the coprime to $2$ part of $x$ and $y$ respectively. 
Then we have $\log(x_{(2)}^{3r}  y_{(2)}^s ) \le 1481$, $ \log(y_{(2)}^{3s}  x_{(2)}^r ) \le 1207$, $\log(x_{(2)}^r) \le 494$ and $\log(y_{(2)}^s) \le 403$.
\end{prop}
\begin{proof}
The proposition follows from Lemma \ref{4-1-lem-2}, Lemma \ref{4-1-lem-3} and the computation in \cite{code_of_zpzhou}.
\end{proof}

\subsection{Structure of solutions}

We shall analyze the structure of the solution $(x,y,z)$ in more detail.

\begin{lem}  \label{4-2-lem-1}
Let $r \ge 7$, $s\ge 8$ be positive integers, $(x,y,z)$ be a triple of positive coprime integers such that $\delta_r x^r + \delta_s y^s = z^3$, where $\delta_r,\delta_s \in \{\pm 1\}$. 

For each prime number $p\ge 17$, denote [cf. Lemma \ref{4-1-lem-1}]
\small\begin{align*}
a_1(p) &\defeq 6 \cdot \frac{p^2+5p}{p^2+p-12}\cdot (1-\frac{1}{4p}),  \\
a_2(p) &\defeq 6 \cdot \Vol(\mathfrak R_p^{(2)})  + a_1(p)\cdot \log(6), 3 \cdot \Vol(\mathfrak R_p^{(3)}) + a_1(p)\cdot \log(6p) \} .
\end{align*}\normalsize
Then we have $6 < a_1(p) < 7.6$, and we have 
$a_2(17) < 403$, $a_2(19) < 425$, $a_2(23) < 472$, 
$a_2(29) < 544$, $a_2(31) < 578$.

Let $l\ge 17$ be a prime number such that $l\nmid rs$.
Write uniquely \small\begin{gather}  \label{4-2-eq1}
x = x_1 \cdot 2^{r_2} \cdot 3^{r_3} \cdot l^{r_l} \cdot x_l^{l}, \;
y = y_1 \cdot 2^{s_2} \cdot 3^{s_3} \cdot l^{s_l} \cdot y_l^{l},
\end{gather}\normalsize
where $x_1, y_1, x_l, y_l, 2, 3, l$ are pairwise coprime positive integers, $r_2, s_2, r_3, s_3, r_l, s_l \ge 0$, $x_l^l = x_{(6l)}^{\lagl}$, $y_l^l = y_{(6l)}^{\lagl}$.

(i) We have $3\nmid rs$. If $r = 7$, then $s \ge 11$. 
Moreover, we have
\begin{equation}\label{4-2-eq2}
\begin{aligned}
3r \cdot \log(x_1) + s \cdot \log(y_1) \le a_2(l) + a_1(l)\cdot \log\rad(x_1 y_1 x_l y_l), \\
r \cdot \log(x_1) + 3s \cdot \log(y_1) \le a_2(l) + a_1(l)\cdot \log\rad(x_1 y_1 x_l y_l).
\end{aligned}
\end{equation}
Hence 
{\small
\begin{equation}\label{4-2-eq2b}
\begin{aligned}
(3r-7.6) \cdot \log(x_1) \le 3r \cdot \log(x_1) - a_1(l) \cdot \log\rad(x_1) \le a_2(l) +  a_1(l) \log(x_l y_l),  \\
(\frac{20s}{7}-7.6) \cdot \log(y_1) \le \frac{20s}{7}\cdot \log(y_1) - a_1(l) \cdot \log\rad(y_1) \le a_2(l) +  a_1(l) \log(x_l y_l).
\end{aligned}
\end{equation} }

(ii) We have $x_l = 1$, $y_l = 1$, hence we have  $a_1(l) \cdot \log(x_l y_l) = 0$ in (\ref{4-2-eq2b}).

(iii) Let $p\ge 5$, $p\neq l$ be a prime number. 
If $l \nmid v_p(x)$, then we have $(3r\cdot v_p(x) - a_1(l))\cdot  \log(p) < a_2(l)$;
if $l \nmid v_p(y)$, then we have $(\frac{20}{7}\cdot s \cdot v_p(x) - a_1(l)) \cdot  \log(p) < a_2(l)$;
if $3\mid x$ and $l\nmid (3r_3r-3)$, then we have $(3r_3r - 3 - a_1(l))\cdot  \log(3) < a_2(l)$;
if $3\mid y$, $l\nmid (s_3s-3)$ and $l\nmid (3s_3s - 1)$, then we have $(\frac{20}{7}  s_3s - 3 - a_1(l))\cdot  \log(3) < a_2(l)$.

(iv) We have $r_l \le \frac{56}{r}$, $s_l \le \frac{56}{s}$.

(v) We have $r_3 \le \frac{153}{r}$, and $r_3 = 0$ when $r > 137$;
$s_3 \le \frac{153}{s}$, and $s_3 = 0$ when $s > 137$.
\end{lem}

\begin{proof}
For assertion (i), since all the signatures $(3,3,n)$for $n\ge 3$, $(3,m,n)$ for $7\le m,n \le 10$ and their permutations have been solved [cf. Introduction], we can deduce that $s\ge 11$, $3\nmid rs$. 
Since $l\nmid rs$, following the notation of Lemma \ref{4-1-lem-1}, we have $x_1^{3r} y_1^s \mid N'_l$, $\log\rad(N) \le \log\rad(x_1 y_1 x_l y_l) + \log(6)$ when $l\nmid xy$, and $\log\rad(N) \le \log\rad(x_1 y_1 x_l y_l) + \log(6l)$ when $l\mid xy$. Hence the first inequality in (\ref{4-2-eq2}) follows from (\ref{4-1-eq2}). Then since $s > 8 > a_1(l)$, we have $s\cdot \log(y_1) \ge a_1(l) \cdot \log(y_1)$, hence $a_2(l) +  a_1(l) \log(x_l y_l) \ge 3r \cdot \log(x_1) - a_1(l) \cdot \log\rad(x_1) \ge (3r-7.6) \cdot \log(x_1)$.

By symmetry, if we ``interchange''  $(x,r)$ and $(y,s)$ in Lemma \ref{4-1-lem-1}, then the second inequality in (\ref{4-2-eq2}) follows from (\ref{4-1-eq2}). If $r\ge 8$, then the second inequality in (\ref{4-2-eq2b}) can be proved by symmety. Suppose that $r=7$, then by computing $\frac{1}{14}\times \text{``the first line of  (\ref{4-2-eq2})''} + \frac{13}{14}\times \text{``the second line of  (\ref{4-2-eq2})''}$, 
we have \small\begin{align*}
8 \cdot \log(x_1) + \frac{20s}{7} \cdot \log(y_1) \le a_2(l) + a_1(l)\cdot \log\rad(x_1 y_1 x_l y_l).
\end{align*}\normalsize
Then assertion (i) follows.

For assertion (ii), recall that we have $\log(x_1 3^{r_3} l^{r_l }x_l^l) = \log(x_{(2)})\le \frac{494}{r}$, hence $x_l \le \exp(\frac{494}{rl} ) \le \exp(\frac{494}{7\cdot 17} ) < 64$. Since $2,3\nmid x_l$, we have $x_l \le 61$.
By symmetry, we can prove that $y_l \le 61$, and note that this is ture for any prime number $l\ge 17$.

When $r \ge 11$, we have $\log(x_l) \le \frac{494}{100 l} < \log(3)$, hence $x_l = 1$. Now suppose that $r < 100$, and by replcaing $s$ by its prime divisor $\ge 11$, we can assume that $s$ is divided by at most one prime number $\ge 11$.

Let $q$ be the smallest prime number $q\ge 17$, such that $q\neq l$, $q\nmid rs$ and $q \nmid x_l$. Then by the above analysis, at most four prime numbers $p\ge 17$ are excluded, hence we have $q \le 31$, $a_2(q) < 578$.
Write uniquely $x = x'_1 \cdot 2^{r_2} \cdot 3^{r_3}\cdot q^{r_q} \cdot x_q^{q}$, $y = y'_1 \cdot 2^{s_2} \cdot 3^{s_3}\cdot q^{s_q} \cdot y_q^{q}$ similar to (\ref{4-2-eq1}), where $x'_1, y'_1, x_q, y_q, 2, 3, q$ are pairwise coprime, $r_2, s_2, r_3, s_3, r_q, s_q \ge 0$, $x_q^q = x_{(6q)}^{\lagk{q}}$, $y_q^q = y_{(6q)}^{\lagk{q}}$.
Since we have $x_q,y_q \le 61$ [by replacing $x_l,y_l$ by $x_q, y_q$], then by  (\ref{4-2-eq2b}) we have $(3r - 7.6) \cdot \log\rad(x'_1) \le a_2(q) + 7.6 \cdot \log(61^2) < a_2(q) + 63 < 641$.

Note that we have $\gcd(x_l, 6q) = 1$, and we claim that $\gcd(x_l, x_q) = 1$. 
In fact, suppose that we have $p \mid x_l$, $p\mid x_q$ for some prime $p$, then $p\ge 5$ [since $2,3 \nmid x_l$] and $ql \mid v_p(x)$ by definition. 
Hence by Proposition \ref{4-1-prop}, we have $494 > \log(x_{(2)}^2) \ge r\cdot v_p(x)\cdot \log(p) \ge rql \cdot\log(5) > 7\cdot 11\cdot 17\cdot \log(5) > 494$ --- a contradiction! Hence we have $\gcd(x_l, x_q) = 1$.

Since $x_l^l \mid x = x'_1 \cdot 2^{r_2}\cdot 3^{r_3} \cdot q^{r_q} \cdot x_q^{q}$ and $\gcd(x_l, 6 q x_q) = 1$, we can deduce that $x_l^l \mid x'_1$. 
Then by (i), $641 > a_2(q)+a_1(q)\cdot\log(x_q y_q) \ge 3r\cdot \log(x'_1) - 4\log\rad(x'_1)  \ge (3rl-7.6) \log(x_l)$, thus 
\small\begin{align}  \label{4-2-eq3}
\log(x_l) \le \frac{ a_2(q)+a_1(q)\cdot\log(x_q y_q)}{3r l-7.6} 
< \frac{641}{3r l-7.6}.
\end{align}\normalsize
Hence when $r \ge 8$, we have $\log(x_l) \le \frac{ 641}{3\cdot 8\cdot 17 - 7.6} < \log(5)$, hence $x_l = 1$; then by the symmetry of $r,s$ [since both of them $\ge 8$, we can interchange $(x,r)$ and $(y,s)$, we shall use the word ``symmetry'' to represent this observasion], we have $y_l = 1$.

Now assume that $r = 7$, then $s\ge 11$. We can similar prove that $x_l = 1$ when $l\ge 23$; $x_l \in \{1,5\}$ when $l\in \{17,19\}$. Since $s\ge 11$, we can also prove similarly that 
\small\begin{align*} 
\log(y_l) \le \frac{ 641}{\frac{20s}{7}\cdot l-7.6}
\le \frac{ 641 }{\frac{20\cdot 11}{7}\cdot 17 - 7.6} < \log(4),
\end{align*}\normalsize
Hence $y_l = 1$. Moreover, when $r = 7$, $x_l \le 5$, we further have $q \le 23$, $a_2(q) < 472$. Hence by (\ref{4-2-eq3}) we have $\log(x_l) < \frac{472 + 7.6 \cdot \log(5)}{3\cdot 7\cdot 17 - 7.6} < \log(4)$, hence $x_l = 1$. This proves assertion (ii).

Aassertion (iii) is a direct consequence of Lemma \ref{4-1-lem-1}). The proof of assertion (iii) is similar to that of assertion (i).

For assertion (iv), let $q$ be the smallest prime number $q\ge 17$, such that $q\neq l$, $q\nmid rs$ and $q \nmid r_l$. 
Write uniquely $x = x'_1 \cdot 2^{r_2} \cdot 3^{r_3}\cdot q^{r_q}$ similar to (\ref{4-2-eq1}), where $2,3,q \nmid x'_1$, $r_2, r_3, r_q \ge 0$ [recall that we have proven $x_q = y_q = 1$ in assertion (ii)].
Then by assertion (iii), we have $(3rr_l - a_1(q))\cdot \log(l) \le a_2(q)$, thus 
\small\begin{align}  \label{4-2-eq4}
r_l \le \frac{a_2(q)/\log(l)+a_1(q)}{3r} .
\end{align}\normalsize

One can easily show that $r_l < 100$ at first.
Then similar to the proof of assertion (ii), in the choice of $q$, we can replace $r,s$ by their prime divisors $\ge 17$, to exclude at most $4$ prime numbers $p\ge 17$, such that $p\mid l rs r_l$. 
Hence $q \le 31$, $a_1(q) < 578$, $r_l \le \frac{578/\log(17) + 7.6 }{3\cdot 7} < 11$, $r_l \le 10$. 
Then we have $q \le 29$, hence 
$$r_l r \le \max_{q\in \{17,19,23,29\}} \frac{a_2(q) / \log(q) + 7.6}{3} < 57,$$
thus $r_l \le \frac{56}{r}$.

For the upper bound of $s_l$, when $r \ge 8$, we have $s_l \le \frac{56}{s}$ by symmetry. Now assume that $r = 7$, $s\ge 11$, And let $q$ be the smallest prime number $q\ge 17$, such that $q\neq l$, $q\nmid rs$ and $q \nmid s_l$. Then we can porve easily that $q\nmid r, s_l$, hence we have $q\le 23$.
Then  
$$s_l s \le \max_{q\in \{17,19,23\}} \frac{a_2(q) / \log(q) + 7.6}{20 / 7} < 56,$$
thus $s_l \le \frac{55}{s} < \frac{56}{s}$.

The proof of assertion (v) is similar to that of assertion (iv), cf. the computation in \cite{code_of_zpzhou} for details.
\end{proof}

\begin{lem}  \label{4-2-lem-2}
Let $r \ge 7$, $s\ge 8$ be positive integers, $(x,y,z)$ be a triple of positive coprime integers such that $\delta_r x^r + \delta_s y^s = z^3$, where $\delta_r,\delta_s \in \{\pm 1\}$. 
Write $r_2 \defeq v_2(x)$, $s_2 \defeq v_2(y)$.  
Then we have $r_2 \le \frac{667}{r}$, $s_2 \le \frac{667}{s}$.
\end{lem}
\begin{proof}
For each prime number $p\ge 17$, write [cf. Lemma \ref{4-1-lem-1}]
\small\begin{gather*}
a_1(p) \defeq 6 \cdot \frac{p^2+5p}{p^2+p-12}\cdot (1-\frac{1}{12p}), \;
a_2(p) \defeq 6 \cdot \Vol(\mathfrak R_p^{(1)}) + a_1(p)\cdot \log(6p).
\end{gather*}\normalsize
Then we have $6 < a_1(p) < 7.6$, and we have 
$a_2(17) < 978$, $a_2(19) < 1041$, $a_2(23) < 1178$, 
$a_2(29) < 1389$, $a_2(31) < 1475$.
Moreover, we have $x_l = y_l = 1$ by Lemma \ref{4-2-lem-1}.

(1) When $r_2 > 0$, let $l \ge 17$ be the smallest prime numbr such that $l\nmid rs$, $l\nmid r_2$. 
Following the notation of Lemma \ref{4-1-lem-1}, we have $2^{3r_2r}x_1^r \mid N'_l$ and $\log\rad(N) \le \log\rad(6l x_1y_1)$.
Then by (\ref{4-1-eq1}), we have $\log(2^{3r_2r} x_1^{3r} y_1^s) \le a_1(p) \cdot \log\rad(6lx_1y_1) + \Vol(\mathfrak R_l^{(1)}) = a_2(p) + 7.6 \cdot \log\rad(x_1 y_1)$. Hence since $s\ge 8$, we have
\begin{align*}
 a_2(l)  &\ge \log(2^{3r_2r} x_1^{3r} y_1^{s} ) - 7.6 \cdot \log\rad(x_1 y_1) \\
     &\ge 3r_2r\cdot \log(2) + \sum_{p\mid x_1y_1} (v_p(x_1^{3r} y_1^{s}) - 7.6)\cdot\log(p)
     \ge  3r_2r\cdot \log(2) .
\end{align*}
Hence \begin{align} \label{4-2-eq5}
 r_2 \le \frac{a_2(l)/\log(2)}{3r} .
\end{align}
By Proposition \ref{4-1-prop}, we have $r_2r\cdot \log(2) \le \log(x_r) \le 1400$, hence $r_2 \le \frac{1400}{7\cdot \log(2)} < 17 \cdot 19$.
Then by replacing $r$ or $s$ to their prime divisors $\ge 17$, we have exclude at most three possible $l \ge 17$, i.e. we have $l\le 29$, $a_2(l) < 1389$. Hence by (\ref{4-2-eq5}) we have $r_2 \le \frac{1389/\log(2)}{3r} < \frac{668}{r}$, $r_2 \le \frac{667}{r}$.

(2) When $s_2 > 0$, if $r\ge 8$, then we have $s_2 \le \frac{667}{s}$ by the symmetry of $r,s$ when $r,s \ge 8$. 
Now assume that $r=7$, then we have $s\ge 11$ by Lemma \ref{4-2-lem-1}. 

Let $l \ge 17$ be the smallest prime numbr such that $l\nmid rs$, $l\nmid s_2$. 
Similar to the proof of (\ref{4-2-eq2b}) in Lemma \ref{4-2-lem-1}, we can prove that $s_2 \le \frac{a_2(l)/\log(2)}{20s/7}$, which is parallel to the proof of (\ref{4-2-eq5}). 
Moreover, since $r = 7 < 17$, we further have $l \le 23$, $a_2(l) < 1178$.
Hence we have $s_2 \le \frac{a_2(l)/\log(2)}{20s/7} \le \frac{1178/\log(2)}{20s/7} < \frac{595}{s} < \frac{667}{s}$.
\end{proof}

\begin{lem} \label{4-2-lem-3}
Let $s \ge r \ge 7$ be positive integers, $(x,y,z)$ be a triple of positive coprime integers such that $\delta_r x^r + \delta_s y^s = z^3$, where $\delta_r,\delta_s \in \{\pm 1\}$. 

For each prime number $p\ge 17$, denote [cf. Lemma \ref{4-1-lem-1}]
\small\begin{align*}
a_1(p) &\defeq 6 \cdot \frac{p^2+5p}{p^2+p-12}\cdot (1-\frac{1}{4p}),  \\
a_2(p) &\defeq 6 \cdot \Vol(\mathfrak R_p^{(2)})  + a_1(p)\cdot \log(6), 3 \cdot \Vol(\mathfrak R_p^{(3)}) + a_1(p)\cdot \log(6p) \} .
\end{align*}\normalsize
Then we have $6 < a_1(p) < 7.6$, and we have 
$a_2(17) < 403$, $a_2(19) < 425$, $a_2(23) < 472$, 
$a_2(29) < 544$, $a_2(31) < 578$.

Let $l\ge 17$ be a prime number such that $l\nmid rs$.
Write uniquely \small\begin{gather*}
x = x_1 \cdot 2^{r_2} \cdot 3^{r_3} \cdot l^{r_l} \cdot x_l^{l}, \;
y = y_1 \cdot 2^{s_2} \cdot 3^{s_3} \cdot l^{s_l} \cdot y_l^{l},
\end{gather*}\normalsize
where $x_1, y_1, x_l, y_l, 2, 3, l$ are pairwise coprime positive integers, $r_2, s_2, r_3, s_3, r_l, s_l \ge 0$, $x_l^l = x_{(6l)}^{\lagl}$, $y_l^l = y_{(6l)}^{\lagl}$.

(i) We have $x_l = 1$, $y_l = 1$;
$r_2 \le \frac{667}{r}$, $s_2 \le \frac{667}{s}$;
$r_3 \le \frac{153}{r}$, $s_3 \le \frac{153}{s}$;
$r_l \le \frac{56}{r}$, $s_l \le \frac{56}{s}$.
$7\le r \le 667$, and we have $x = 2^{r_2}$ when $r \ge 138$;
$11\le s \le 667$, and we have $y = 2^{s_2}$ when $s \ge 138$.

(ii) We have 
\begin{align*}
(3r-a_1(l)) \cdot \log(x_1) + (s-a_1(l))\cdot \log(y_1) \le a_2(l), \\
(r-a_1(l)) \cdot \log(x_1) + (3s-a_1(l))\cdot \log(y_1) \le a_2(l).
\end{align*}

(iii) When $3r \ge s$, we have \begin{align*}
(\frac{4rs}{r+s} - a_1(l) ) \cdot \log(x_1 y_1)  \le a_2(l).
\end{align*}

\end{lem}
\begin{proof}
Assertions (i) and (ii) follow from Lemma \ref{4-2-lem-1} and Lemma \ref{4-2-lem-2}. Note that when $r \ge 138$, we have $r_3 = r_l = 0$. 
To show $x = 2^{r_2}$, it suffices to prove that $x_1 = 1$ for some $l\ge 17$, $l\nmid rs$. We can assume that each of $r,s$ is divided by at most one prime number $\ge 17$, then we can take $l\in\{17,19,23\}$. Hence $a_2(l) < 472$.
Then $\log(x_1) < \frac{a_2(l)}{3r - 7.6} < \frac{472}{3\cdot 138 - 7.6} < \log(4)$, hence $x_1 = 1$. Then $x = 2^{r_2}$. If $r\ge 668$, then $r_2 = 0$, $x=1$, which is impossible by Catalan's conjecture! Hence $r \le 667$. The results concerning $s$ can be proved similarly.

assertion (iii) is obtained by computing $(1- \lambda )\times \text{``the first inequality of (ii)''} + \frac{13}{14}\times \text{``the second inequality of (ii)''}$, where $\lambda = \frac{3r-s}{2(r+s)}$.
\end{proof}

\begin{prop} \label{4-2-prop}
For any positive integers $(r,s)$, such that $r \ge 12$, $s\ge 30$, or $s > r\ge 18$, there does not exist any triple $(x, y, z)$ of positive coprime integers satisfying  $\delta_r x^r + \delta_s y^s =  z^3$, where $\delta_r,\delta_s \in \{\pm 1\}$. 
\end{prop}
\begin{proof}
The proof of the proposition follows from the computation in \cite{Diophantine_after_IUT_I}, where we have $s \neq r$ since the signatures $(3,n,n)$ with $n\ge 3$ are proven.
The algorithm we used is similar to the algorithm described in Remark \ref{2-2-rmk-1}, which is based on Lemma \ref{4-2-lem-3}.
\end{proof}

\begin{cor} \label{4-2-cor}
Suppose that $(r,s,t)$ is a permutation of $(3,m,n)$ with $n\ge m \ge 3$, such that $m, n$ does not satisfy one of the following conditions:
\begin{itemize}
\item $m\in \{4,8,10\}$, and  $11\le n \le 109$ or $n\in \{113, 121\}$.
\item $m = 5$, and $7 \le n \le 3677$; $m\in \{7, 11\}$, and $11 \le n \le 667$.
\item $13 \le m \le 17$, $m < n \le 29$.
\end{itemize}
Then the generalized Fermat equation $x^r + y^s = z^t$ has no non-trivial primitive solution except for the solutions related to the Catalan solutions and the nine non-Catalan solutions enumerated in the introduction.
\end{cor}
\begin{proof}
cf. Corollary \ref{3-2-cor}, Proposition \ref{4-1-prop} and Proposition \ref{4-2-prop}.
\end{proof}
\section{Conclusion}
In the last three sections, we have researched on the generalized Fermat equation \begin{equation} \label{6-GenFE} \tag{$\star$}
x^r + y^s = z^t, \;\text{with}\; x, y, z\in\Z
\end{equation}
with signature $(r,s,t)$ in three different cases. 
As a summary, we have the following result.

\begin{thm} \label{5-thm1}
Let $r,s,t\ge 2$ be positive integers such that $\frac{1}{r} + \frac{1}{s} + \frac{1}{t} \le 1$. 
Then the generalized Fermat equation (\ref{6-GenFE}) has no non-trivial primitive solution, exceplt for the solutions related to the Catalan solution $1^n+2^3 = 3^2$ and nine non-Catalan solutions listed in the Introduction, when $(r,s,t)$ is not a permutation of the following signatures:
\begin{itemize}
\item $(4,5,n)$, $(4,7,n)$, $(5,6,n)$, with $7 \le n \le 303$.
\item $(2,3,n)$, $(3,4,n)$, $(3,8,n)$, $(3,10,n)$, with $11\le n \le 109$ or $n\in \{113, 121\}$.
\item $(3,5,n)$, with $7\le n \le 3677$; $(3,7,n)$, $(3,11,n)$, with $11 \le n \le 667$.
\item $(3,m,n)$, with $13 \le m \le 17$, $m < n \le 29$; $(2,m,n)$, with $m \ge 5$, $n\ge 7$.
\end{itemize} 
\end{thm}
\begin{proof}
The theorem follows from Propositions \ref{2-2-prop}, Corollary \ref{3-2-cor}, and Corollary \ref{4-2-cor}.
\end{proof}

\begin{tiny-remark} \label{5-rmk1}
(i) It is worth noting that with further analysis of upper bounds and more computations for searching solutions, an expanded set of signatures including the permutations of $(4, 7, n)$, $(5, 6, n)$ are likely to be solved.

(ii) Similar to the methods in Section 4, for the generalized Fermat equations with signature $(2,m,n)$ and thier permutations, since the signature $(2,4,n)$, $(2,n,4)$, $(2,5,5)$ have been solved, we can deduce that $m,n \ge 5$ and $\max\{m,n\} \ge 7$.

By using Lemma \ref{1-2-cor4}, we can prove upper bounds for $m,n$ similarly. A very rough estimation shows that we have $n < 2600$ for $m = 5$, and $n < 1200$ for $m\ge 7$, the related work will be undertaken in future studies.

(iii) Some of the signatures listed above can be exculded based on the proven signatures.
For example, since the signatures $(2,4,5)$, $(4,5,2)$, $(4,5,3)$, $(2,5,5)$ have been solved [cf. Introduction], permutations of $(4,5,n)$ with $n$ a multiple of $2,3,5$ can also be solved, can the related generalized Fermat equation (\ref{6-GenFE}) has no non-trivial primitive solution in this case; similarly, we can show that for permutations of $(4,7,n)$ with $n$ a multiple of $2,3,7$, and permutations of $(5,6,n)$ with $n$ a multiple of $2,3,5$, (\ref{6-GenFE}) has no non-trivial primitive solution.
\end{tiny-remark}

\begin{cor} \label{5-cor1}
To solve the generalized Fermat equation $x^r + y^s = z^t$ with exponents $r,s,t \ge 4$, we are left with $244$ signatures $(r,s,t)$ up to permutation;
to solve the Beal conjecture, we are left with $2446$ signatures $(r,s,t)$ up to permutation.
\end{cor}
\begin{proof}
For the case of $r,s,t\ge 4$, cf. the discussion in Remark \ref{5-rmk1}, (iii). 
For the case of Beal conjecture and the computation of the number of unsolved signatures, cf. the computation in \cite{code_of_zpzhou}.
\end{proof}

\bibliographystyle{plain}
\bibliography{zpzhou.bib}

\begin{thebibliography}{10}

\bibitem{TorsionPointFields}
Clemens Adelmann.
\newblock {\em The decomposition of primes in torsion point fields}, volume
  1761 of {\em Lecture Notes in Mathematics}.
\newblock Springer-Verlag, Berlin, 2001.

\bibitem{semi_local_approximation_fermat_catalan_popular}
Boris Bartolomé and Preda Mih\u{a}ilescu.
\newblock Semilocal approximation for the fermat-catalan and further popular
  diophantine norm equations, 2022.

\bibitem{Bennett-Chen}
Michael~A. Bennett and Imin Chen.
\newblock Multi-{F}rey {$\Bbb Q$}-curves and the {D}iophantine equation
  {$a^2+b^6=c^n$}.
\newblock {\em Algebra Number Theory}, 6(4):707--730, 2012.

\bibitem{Bennett-Chen-Dahmen-Yazdani}
Michael~A. Bennett, Imin Chen, Sander~R. Dahmen, and Soroosh Yazdani.
\newblock Generalized {F}ermat equations: a miscellany.
\newblock {\em Int. J. Number Theory}, 11(1):1--28, 2015.

\bibitem{Bennett-Ellenberg-Ng}
Michael~A. Bennett, Jordan~S. Ellenberg, and Nathan~C. Ng.
\newblock The {D}iophantine equation {$A^4+2^\delta B^2=C^n$}.
\newblock {\em Int. J. Number Theory}, 6(2):311--338, 2010.

\bibitem{Bennett-Skinner-8}
Michael~A. Bennett and Chris~M. Skinner.
\newblock Ternary {D}iophantine equations via {G}alois representations and
  modular forms.
\newblock {\em Canad. J. Math.}, 56(1):23--54, 2004.

\bibitem{spherical_case1}
Frits Beukers.
\newblock The {D}iophantine equation {$Ax^p+By^q=Cz^r$}.
\newblock {\em Duke Math. J.}, 91(1):61--88, 1998.

\bibitem{Magma}
Wieb Bosma, John Cannon, and Catherine Steel.
\newblock Magma: A system for algebra and geometry computation.
\newblock {\em Journal of Symbolic Computation}, 24(3):235--265, 1997.

\bibitem{Brown_case_2_3_10}
David Brown.
\newblock Primitive integral solutions to {$x^2+y^3=z^{10}$}.
\newblock {\em Int. Math. Res. Not. IMRN}, (2):423--436, 2012.

\bibitem{Bruin-16}
Nils Bruin.
\newblock The {D}iophantine equations {$x^2\pm y^4=\pm z^6$} and
  {$x^2+y^8=z^3$}.
\newblock {\em Compositio Math.}, 118(3):305--321, 1999.

\bibitem{Bruin-18}
Nils Bruin.
\newblock Chabauty methods using elliptic curves.
\newblock {\em J. Reine Angew. Math.}, 562:27--49, 2003.

\bibitem{Bruin-19}
Nils Bruin.
\newblock The primitive solutions to {$x^3+y^9=z^2$}.
\newblock {\em J. Number Theory}, 111(1):179--189, 2005.

\bibitem{Chen-21}
Imin Chen.
\newblock On the equation {$s^2+y^{2p}=\alpha^3$}.
\newblock {\em Math. Comp.}, 77(262):1223--1227, 2008.

\bibitem{Chen-Siksek}
Imin Chen and Samir Siksek.
\newblock Perfect powers expressible as sums of two cubes.
\newblock {\em J. Algebra}, 322(3):638--656, 2009.

\bibitem{Cohen-GTM240}
Henri Cohen.
\newblock {\em Number theory. {V}ol. {II}. {A}nalytic and modern tools}, volume
  240 of {\em Graduate Texts in Mathematics}.
\newblock Springer, New York, 2007.

\bibitem{Dahmen-29}
Sander~R. Dahmen.
\newblock A refined modular approach to the {D}iophantine equation
  {$x^2+y^{2n}=z^3$}.
\newblock {\em Int. J. Number Theory}, 7(5):1303--1316, 2011.

\bibitem{Dahmen-Siksek}
Sander~R. Dahmen and Samir Siksek.
\newblock Perfect powers expressible as sums of two fifth or seventh powers.
\newblock {\em Acta Arith.}, 164(1):65--100, 2014.

\bibitem{Darmon-Granville}
Henri Darmon and Andrew Granville.
\newblock On the equations {$z^m=F(x,y)$} and {$Ax^p+By^q=Cz^r$}.
\newblock {\em Bull. London Math. Soc.}, 27(6):513--543, 1995.

\bibitem{Darmon-Merel}
Henri Darmon and Lo\"ic Merel.
\newblock Winding quotients and some variants of {F}ermat's last theorem.
\newblock {\em J. Reine Angew. Math.}, 490:81--100, 1997.

\bibitem{spherical_case2}
Johnny Edwards.
\newblock A complete solution to {$X^2+Y^3+Z^5=0$}.
\newblock {\em J. Reine Angew. Math.}, 571:213--236, 2004.

\bibitem{Ellenberg}
Jordan~S. Ellenberg.
\newblock Galois representations attached to {$\Bbb Q$}-curves and the
  generalized {F}ermat equation {$A^4+B^2=C^p$}.
\newblock {\em Amer. J. Math.}, 126(4):763--787, 2004.

\bibitem{Beal_Conj}
R.~Daniel Mauldin.
\newblock A generalization of {F}ermat's last theorem: the {B}eal conjecture
  and prize problem.
\newblock {\em Notices Amer. Math. Soc.}, 44(11):1436--1437, 1997.

\bibitem{Mihilescu2004PrimaryCU}
Preda Mih\u{a}ilescu.
\newblock Primary cyclotomic units and a proof of {C}atalan's conjecture.
\newblock {\em J. Reine Angew. Math.}, 572:167--195, 2004.

\bibitem{IUTchI}
Shinichi Mochizuki.
\newblock Inter-universal {T}eichm\"uller theory {I}: {C}onstruction of {H}odge
  theaters.
\newblock {\em Publ. Res. Inst. Math. Sci.}, 57(1-2):3--207, 2021.

\bibitem{IUTchII}
Shinichi Mochizuki.
\newblock Inter-universal {T}eichm\"uller theory {II}:
  {H}odge-{A}rakelov-theoretic evaluation.
\newblock {\em Publ. Res. Inst. Math. Sci.}, 57(1-2):209--401, 2021.

\bibitem{IUTchIII}
Shinichi Mochizuki.
\newblock Inter-universal {T}eichm\"uller theory {III}: {C}anonical splittings
  of the log-theta-lattice.
\newblock {\em Publ. Res. Inst. Math. Sci.}, 57(1-2):403--626, 2021.

\bibitem{IUTchIV}
Shinichi Mochizuki.
\newblock Inter-universal {T}eichm\"uller theory {IV}: {L}og-volume
  computations and set-theoretic foundations.
\newblock {\em Publ. Res. Inst. Math. Sci.}, 57(1-2):627--723, 2021.

\bibitem{ExpEst}
Shinichi Mochizuki, Ivan Fesenko, Yuichiro Hoshi, Arata Minamide, and Wojciech
  Porowski.
\newblock Explicit estimates in inter-universal {T}eichm\"uller theory.
\newblock {\em Kodai Math. J.}, 45(2):175--236, 2022.

\bibitem{Poonen}
Bjorn Poonen.
\newblock Some {D}iophantine equations of the form {$x^n+y^n=z^m$}.
\newblock {\em Acta Arith.}, 86(3):193--205, 1998.

\bibitem{Poonen-Schaefer-Stoll}
Bjorn Poonen, Edward~F. Schaefer, and Michael Stoll.
\newblock Twists of {$X(7)$} and primitive solutions to {$x^2+y^3=z^7$}.
\newblock {\em Duke Math. J.}, 137(1):103--158, 2007.

\bibitem{Serre1971PropritsGD}
Jean-Pierre Serre.
\newblock Propri\'et\'es galoisiennes des points d'ordre fini des courbes
  elliptiques.
\newblock {\em Invent. Math.}, 15(4):259--331, 1972.

\bibitem{Siksek-Stoll}
Samir Siksek and Michael Stoll.
\newblock Partial descent on hyperelliptic curves and the generalized {F}ermat
  equation {$x^3+y^4+z^5=0$}.
\newblock {\em Bull. Lond. Math. Soc.}, 44(1):151--166, 2012.

\bibitem{Siksek-Stoll_case_2_3_15}
Samir Siksek and Michael Stoll.
\newblock The generalised {F}ermat equation {$x^2+y^3=z^{15}$}.
\newblock {\em Arch. Math. (Basel)}, 102(5):411--421, 2014.

\bibitem{Silverman2009EllipticCurves}
Joseph~H. Silverman.
\newblock {\em The arithmetic of elliptic curves}, volume 106 of {\em Graduate
  Texts in Mathematics}.
\newblock Springer, Dordrecht, second edition, 2009.

\bibitem{Taylor-Wiles}
Richard Taylor and Andrew Wiles.
\newblock Ring-theoretic properties of certain {H}ecke algebras.
\newblock {\em Ann. of Math. (2)}, 141(3):553--572, 1995.

\bibitem{Wiles}
Andrew Wiles.
\newblock Modular elliptic curves and {F}ermat's last theorem.
\newblock {\em Ann. of Math. (2)}, 141(3):443--551, 1995.

\bibitem{code_of_zpzhou}
Zhong-Peng Zhou.
\newblock {Github Code Repository for ``The inter-universal Teichm\"uller
  theory and new Diophantine results over the rational numbers. II'' }.
\newblock Available at \url{https://github.com/zhongpengzhou/IUT-Q-II-code}.

\bibitem{Diophantine_after_IUT_I}
Zhong-Peng Zhou.
\newblock The inter-universal {T}eichm\"uller theory and new {D}iophantine
  results over the rational numbers. {I}.
\newblock Available as a preprint at \url{https://arxiv.org/abs/2503.14510},
  2025.

\end{thebibliography}

\Addresses

\end{document}